\documentclass[12pt]{amsart}
\usepackage{mathrsfs}
\usepackage{amsfonts}
\usepackage{amsmath}
\usepackage{amssymb}
\usepackage{graphicx}
\usepackage{color}
\usepackage{xcolor}
\usepackage{amsthm,amscd}
\usepackage{calc}
\usepackage{hyperref}
\usepackage{setspace}
\usepackage{bm}
\usepackage{enumitem}
\setlist[enumerate]{label=(\roman*)}

\newtheorem{thm}{Theorem}[section]
\newtheorem{cor}[thm]{Corollary}
\newtheorem{lem}[thm]{Lemma}

\newcommand{\Z}{\mathbb Z}
\newcommand{\R}{\mathbb R}

\newcommand{\n}{\nabla}

\newcommand{\pd}{\partial}

\newcommand{\vH}{\mathbf{H}}

\newcommand{\vx}{\mathbf{x}}
\newcommand{\vv}{\mathbf{v}}

\newtheorem{theorem}{Theorem}

\newtheorem{remark}[theorem]{Remark}

\title[]{Topology of gradient Ricci shrinkers via weighted $L^2$ cohomology}

\author{Fei He}
\email{hefei@xmu.edu.cn}
\address{School of Mathematical Science, Xiamen University, Xiamen, China 361005}
\begin{document}

\begin{abstract}
This paper proves several topological results for smooth gradient Ricci shrinkers. We establish upper bounds for the Betti numbers, a vanishing theorem for cohomology, and a dichotomy for the number of ends. We also prove a full Hodge theorem for a large class of shrinkers. The methods are based on weighted $L^2$ cohomology and extend to self-shrinkers of the mean curvature flow.
\end{abstract}

\maketitle

\section{Introduction}
A gradient Ricci shrinker is a triple $(M, g, f)$, where $(M, g)$ is a complete Riemannian manifold and $f$ is a smooth function, such that the following equation holds:
\begin{equation}\label{eqn: Ricci shrinker equation}
\operatorname{\operatorname{Ric}} + \nabla\nabla f = \frac{1}{2} g,
\end{equation}
where $\operatorname{\operatorname{Ric}}$ denotes the Ricci curvature tensor and $\nabla\nabla f$ denotes the Hessian of $f$. We can, and will, normalize the potential function $f$ so that $|\nabla f|^2 + \operatorname{S} = f$, where $\operatorname{S}$ denotes the scalar curvature.

Gradient Ricci shrinkers arise naturally in the study of Ricci flow, as they are expected to model finite-time singularities of non-collapsed solutions. This expectation has been confirmed for dimensions $2$ and $3$ \cite{Hamilton95,Perelman02}, for Type I singularities in all dimensions \cite{CaoZhang11,EMT11,Naber10,Sesum06},  and, in a certain weak sense, in the general case \cite{Bamler2020, Bamler2023}. Consequently, understanding the topology of gradient Ricci shrinkers is of fundamental importance, and this has been a very active area of research over the last couple of decades. For example, the geometry and topology of gradient Ricci shrinkers have been extensively studied by Munteanu and Wang in a series of works \cite{MW2015,MW2015b,MW2017a,MW2017b,MW2019}, which have significantly enhanced our understanding of the subject. Recently, K\"ahler Ricci shrinkers in complex dimension $2$ have been completely classified, see \cite{BCCD, CCD, CDS, LW2024} and references therein. At the same time, non-K\"ahler Ricci shrinkers are known to exist \cite{AK22}. A complete classification of gradient Ricci shrinkers remains out of reach at the moment. 

Using tools from algebraic geometry, Sun and Zhang~\cite{SunZhang24} established that every gradient K\"ahler Ricci shrinker has finite topological type. It is natural to ask whether this property extends to non-K\"ahler Ricci shrinkers as well. However, such an expectation may be overly optimistic. The symplectic and algebraic structures that played a critical role in the study of K\"ahler Ricci shrinkers are simply not available in the non-K\"ahler case. Nevertheless, it remains desirable to develop tools for effectively estimating topological invariants of gradient Ricci shrinkers. These tools should not rely on the K\"ahler structure. This is the primary goal of the present article.

Our first result is an estimate of the Betti numbers of gradient Ricci shrinkers. Let
\[
\lambda_1(\operatorname{\operatorname{Ric}}) \leq \lambda_2(\operatorname{\operatorname{Ric}}) \leq \cdots \leq \lambda_n(\operatorname{\operatorname{Ric}})
\]
denote the eigenvalues of the Ricci tensor. For simplicity, we implicitly assume that all manifolds under consideration are smooth and \textbf{orientable}. In the non-orientable case, one may lift to the orientable double cover and apply the results there.
\begin{thm}\label{thm-intro:  Betti numbers of gradient Ricci shrinkers}
For any constants $\epsilon_0> 0$ and $c_0 \geq 0$, there is a constant $C_1$ depending on $\epsilon_0$, $c_0$ and the dimension $n$, such that the following holds.
Let $(M^n, g, f)$ be a complete gradient Ricci shrinker satisfying
\begin{enumerate}
\item
$\sum_{k=n-p+1}^{n} \lambda_{k}(\operatorname{Ric}) \leq  \frac{1}{4}\operatorname{S} + \left( \frac{1}{4} - \epsilon_0\right) f + c_0$ on $M$;
\item there is a constant $k$ such that $\langle \mathcal{R}(\omega), \omega \rangle \geq - k |\omega|^2$ for every $p$-form $\omega$ on the domain $\{f < 2 C_1\}$, where $\mathcal{R}$ is the curvature term in the Hodge Laplacian, see \eqref{eqn: curvature term in Hodge Laplacian}.
\end{enumerate}
Then the Betti number $b_p(M)$ satisfies
\[
b_p(M) \leq C e^{-2\mu/n}  \binom{n}{p},
\]
where $\mu = \ln \int e^{-f} (4\pi)^{-\frac{n}{2}} dv$, and $C$ is a constant depending only on $n, \epsilon_0, c_0$, and $k$.
\end{thm}
The constant $\mu$ in Theorem~\ref{thm-intro: Betti numbers of gradient Ricci shrinkers} is Perelman's entropy, which enters via the sharp Sobolev inequality of Y.~Li and B.~Wang~\cite{LW2020}. We note that gradient Ricci shrinkers always have vanishing first Betti number, since their first fundamental group is finite~\cite{Wyl08}. Theorem~\ref{thm-intro: Betti numbers of gradient Ricci shrinkers} follows from Theorem~\ref{thm: finite Betti numbers}, which is established in the broader setting of smooth metric measure spaces. For closed Riemannian manifolds, Betti number estimates of this type were obtained by P.~Li~\cite{Li80}; our focus is therefore primarily on the noncompact case.
\begin{remark}
Note that Condition $(i)$ is automatically satisfied if $|\operatorname{Ric}| = o(f)$ as $f\to \infty$. In general, by moving $\frac{1}{4}\operatorname{S}$ to the left-hand side, we can rewrite condition $(i)$ as
\[
\frac{3}{4}\sum_{k=n-p+1}^{n} \lambda_{k}(\operatorname{Ric}) - \frac{1}{4} \sum_{k = 1}^{n-p} \lambda_k (\operatorname{Ric}) \leq  \left( \frac{1}{4} - \epsilon_0\right) f + c_0,
\]
where the right-hand side grows quadratically by~\cite{CZ2010}.Intuitively, this ``pinching-growth'' condition ensures that the large eigenvalues of $\operatorname{Ric}$ do not grow disproportionately fast relative to the smaller ones.
Since $\operatorname{S} \le \frac{n}{p} \sum_{k=n-p+1}^n \lambda_k(\operatorname{Ric})$, 
condition $(i)$ implies 
$$\operatorname{S}  \leq \frac{n}{p}\left( \operatorname{S}/4 + (1/4 - \epsilon_0)f + c_0 \right).
$$
For $p \geq \frac{n}{2}$, this yields $\operatorname{S} < f$ when $f$ is large, which via $|\nabla f|^2 + \operatorname{S} = f$ precludes critical points of $\nabla f$ outside a compact set, guaranteeing finite Betti numbers (though without quantitative bounds). For $p < \frac{n}{2}$, however, faster than quadratic curvature growth is not ruled out, and finiteness of the Betti numbers is not known a priori. Theorem~\ref{thm-intro: Betti numbers of gradient Ricci shrinkers} is therefore most significant when $p < \frac{n}{2}$.
\end{remark}

Since most known examples of gradient Ricci shrinkers are noncompact, an important question is how many ends they can have. Munteanu and Wang \cite{MW2015} speculated that a gradient Ricci shrinker has either only one end or splits as a product. They confirmed their speculation for gradient K\"ahler Ricci shrinkers \cite{MW2015}; they also proved that a gradient Ricci shrinker must have only one end if $\operatorname{S} \leq \frac{n}{3}$ \cite{MW2022}. Munteanu, Schulze, and Wang \cite{MSW21} obtained an estimate of the number of ends for gradient Ricci shrinkers under some integral curvature assumptions. See also \cite{BB2026}, \cite{HW2022}, and \cite{QW2022} concerning the number of ends. In general, it remains open whether a gradient Ricci shrinker can have infinitely many ends. We obtain the following result concerning the number of ends.

\begin{thm}\label{thm-intro: number of ends}
Let $(M^n, g, f)$ be a complete noncompact gradient Ricci shrinker satisfying \eqref{eqn: Ricci shrinker equation}. If it satisfies
\[
\operatorname{Ric} - \frac{1}{2}\operatorname{S}g \geq - \frac{n-1}{4}g,
\]
then one of the following cases must hold:
\begin{enumerate}
\item $b_{n-1}(M) = 0$ and $(M, g)$ has only one end;
\item $b_{n-1}(M) = 1$ and $(M, g)$ splits as $(N, g_N) \times (\mathbb{R}, g_{\mathrm{Eucl}})$, where $(N, g_N)$ is a compact Einstein manifold.
\end{enumerate}
\end{thm}
Thus we confirm Munteanu and Wang's speculation under this particular lower bound assumption on the Einstein tensor $\operatorname{Ric} - \frac{1}{2}\operatorname{S}g$. This assumption is satisfied, for instance, by product shrinkers $N^{n-k}\times \R^k$, where $N^{n-k}$ is Einstein, and $k \geq 1$. From the proofs of Theorems \ref{thm-intro:  Betti numbers of gradient Ricci shrinkers} and \ref{thm-intro: number of ends}, we see that the lower bound of the Einstein tensor controls $b_{n-1}(M)$, and hence controls the number of ends for a gradient Ricci shrinker.

To control other Betti numbers, assumptions involving the curvature operator $\operatorname{Rm}$ are needed, see \eqref{eqn: curvature operator} for the definition. We obtain the following vanishing theorem.
\begin{thm}\label{thm-intro: vanishing of betti numbers}
Let $(M^n, g, f)$ be a complete gradient Ricci shrinker satisfying \eqref{eqn: Ricci shrinker equation}, and let $\operatorname{Rm}$ be the curvature operator. For $1\leq p \leq n-1$,
\begin{enumerate}
\item if $(p-1)\operatorname{Rm} < \frac{1}{2 }$, then $b_p(M) = 0$;
\item if $(p-1)\operatorname{Rm} \leq \frac{1}{2 }$, then $b_p(M) \leq \binom{n}{p}$, and $L^2(dv_f)$-integrable $f$-harmonic $p$-forms are parallel.
\end{enumerate}
Moreover, for noncompact $M$, if $\operatorname{Rm} < \frac{1}{2(n-2)}$ and the integral homology groups $H_k(M;\Z)$ are torsion-free for $k \geq 2$, then $M$ is contractible.
\end{thm}

This theorem makes sharp use of the Ricci shrinker equation~\eqref{eqn: Ricci shrinker equation}: the assumptions guarantee that the curvature term in the Weitzenböck formula for the weighted Hodge Laplacian is nonnegative. In the last claim, the non-compactness assumption rules out the compact case $f = \text{constant}$ and $\operatorname{Ric} = \frac{1}{2}g$. (In fact, if $\operatorname{Ric} \geq \frac{1}{2}$ and $\operatorname{Rm} < \frac{1}{2(n-2)}$, writing $\operatorname{Ric}(e_i,e_i) = \sum_{j \neq i} \langle \operatorname{Rm}(e_i \wedge e_j), e_i \wedge e_j\rangle$ in an orthonormal frame shows that the curvature operator is positive, from which the standard topological consequences follow.)

To prove the above theorems, we explore a weighted version of $L^2$-cohomology. For generality, we study it on smooth metric measure spaces; this is not merely a formal abstraction. In fact, we will demonstrate that the tools developed here are also applicable to self-shrinkers of the mean curvature flow (see Theorem \ref{thm: number of ends for MCSSh}). 

Recall that a smooth metric measure space is a triple $(M, g, dv_f := e^{-f} dv)$, where $M$ is a smooth connected manifold of dimension $n$, $g$ is a complete Riemannian metric, and $f$ is a smooth function on $M$. Gradient Ricci shrinkers naturally fall into this category.

The Hodge theorem is a classic tool for studying the topology of compact Riemannian manifolds through analytic methods. However, it does not generally hold for complete noncompact manifolds. A natural remedy is to restrict attention to $L^2$-integrable differential forms, leading to the theories of $L^2$-cohomology and $L^2$-harmonic forms, which were first introduced in \cite{APS75} and have been studied by many authors. Since our aim is not to provide a comprehensive survey, we refer the reader to \cite{Anderson87}, \cite{Carron2002}, and \cite{Dai2011} for excellent introductions to the theory of $L^2$-cohomology on complete Riemannian manifolds.

In general, $L^2$-cohomology groups may be infinite-dimensional; even when they are finite-dimensional, they may depend on the chosen metric. Thus, a fundamental problem is to interpret the topological meaning of $L^2$-cohomology groups. This typically requires precise control of the geometry at infinity. For Ricci shrinkers, although useful information is provided by \cite{MW2015b}, \cite{MW2017b} and \cite{MW2019}, the geometry at infinity remains far from well understood even for some known examples~\cite{BCCD}.

A natural question is whether one can modify the $L^2$-norm by a weight so that the resulting weighted $L^2$-cohomology recovers the ordinary de Rham cohomology. There are situations where this is possible; see, for example, \cite{Bue99}, \cite{Bullock02}, \cite{Ye04}. However, none of the existing theories in the literature are directly applicable to gradient Ricci shrinkers, since they are tailored to different geometric settings. Ahmed and Stroock \cite{AS00} developed a general weighted Hodge theorem for a class of smooth metric measure spaces, but their assumptions rule out Ricci shrinkers (see Remark~\ref{rem: remark about AS00}).

To overcome these difficulties, we develop Agmon-type estimates for weighted harmonic forms, both on $M$ and on sublevel sets of $f$ with Neumann boundary conditions, under assumptions that are tailored for application to shrinkers. These estimates provide growth control for weighted harmonic forms, which allows us to bound a homotopy operator that is critical for establishing the isomorphism between weighted $L^2$-cohomology and de Rham cohomology. We thereby obtain a full Hodge theorem for certain classes of smooth metric measure spaces, see Theorem \ref{thm: identification of reduced L2 cohomology and de Rham cohomology} and Theorem \ref{thm: identification of L2 cohomology and de Rham cohomology}. Theorem \ref{thm-intro: number of ends} and Theorem \ref{thm-intro: vanishing of betti numbers} then follow from this Hodge theorem and topological arguments. Moreover, when the geometry at infinity is potentially complicated and a global Hodge isomorphism is not available, our estimate for weighted harmonic forms with Neumann boundary conditions still yields topological consequences. It enables us to adapt the method of P.~Li \cite{Li80,Li97} to prove uniform upper bounds for the Betti numbers of sublevel sets of $f$, and then topological arguments yield the claim of Theorem~\ref{thm-intro: Betti numbers of gradient Ricci shrinkers}.

We remark that another possible approach is to use Witten's deformed Laplacian on noncompact manifolds \cite{DY23}, but we will not pursue that direction here.

Our weighted Hodge theorem holds under curvature assumptions that are satisfied by a large class of gradient Ricci shrinkers. Let $\mathcal{H}_f^p(M)$ denote the space of $f$-harmonic $p$-forms on $M$ that are $L^2$-integrable with respect to the measure $dv_f$, let $\bar{H}_{(2)}^p(M, dv_f)$ be the reduced weighted $L^2$-cohomology group, see section \ref{section: weighted L2 cohomology} for the definitions, and let $H_{dR}^p(M)$ be the de Rham cohomology group.
\begin{thm}\label{thm-intro: hodge theorem for ricci shrinkers}
Let $(M^n,g, f)$ be a complete gradient Ricci shrinker satisfying $|\operatorname{Ric}| \leq \frac{1}{5n} f$ outside some compact domain, then for each $p$ we have 
\[
\mathcal{H}_f^p(M) \cong \bar{H}^p_{(2)}(M,dv_f) \cong H^p_{dR}(M).
\]
\end{thm}
Note that $f$ has quadratic growth on a gradient Ricci shrinker by the well-known work of Cao and Zhou \cite{CZ2010}. All currently known examples of gradient Ricci shrinkers have bounded curvature and are therefore covered by our result. By \cite{LW2024}, K\"ahler Ricci shrinker surfaces have bounded curvature. By \cite{CFSZ2020}, gradient Ricci shrinker singularity models in dimension $4$ have at most quadratic curvature growth. In fact, understanding the curvature growth (or decay) pattern of gradient Ricci solitons remains an interesting problem; see, for example, \cite{CZ2022,MS2013,MSW2019,MW2015b,MW11}.

The identification of de Rham and weighted $L^2$-cohomology enables us to obtain topological information for these noncompact smooth metric measure spaces via Hodge theory. As applications, we obtain estimates of Betti numbers for smooth metric measure spaces (see Corollary \ref{cor: betti number estimates - R_f}) and for gradient Ricci shrinkers (see Corollary \ref{cor: betti number estiamtes for Ricci shrinkers - R_f}), the latter being a variant of Theorem \ref{thm-intro:  Betti numbers of gradient Ricci shrinkers}.

The article is organized as follows. In Section \ref{section: weighted L2 cohomology}, we introduce $L^2$-cohomology, including the necessary adaptations to the weighted case. In Section \ref{section: growth estimates}, we establish Agmon-type estimates for weighted harmonic forms and prove the finite-dimensionality result, Theorem \ref{thm-intro: finite dimension}. Section \ref{section: l2 cohomology and de Rham cohomology} is devoted to studying the topological interpretation of weighted $L^2$-cohomology on smooth metric measure spaces, Theorem \ref{thm-intro: hodge theorem for ricci shrinkers} is proved at the end of this section. Applications to gradient Ricci shrinkers are presented in Section \ref{section: applications}. For the reader's convenience, we restate Theorem \ref{thm-intro:  Betti numbers of gradient Ricci shrinkers}, Theorem \ref{thm-intro: number of ends}, and Theorem \ref{thm-intro: vanishing of betti numbers} as Theorem \ref{thm: estimate of Betti numbers for gradient Ricci shrinkers}, Theorem \ref{thm: number of ends}, and Theorem \ref{thm: vanishing of cohomology with curvature operator upper bounds} (together with Corollary \ref{cor: contractibility of gradient Ricci shrinkers}), respectively, and their proofs are given there. Applications to mean curvature self-shrinkers are discussed at the end of Section \ref{section: applications}.

\textbf{Acknowledgments:} The author would like to thank Teng Huang, Jianyu Ou, Lihan Wang, Junrong Yan, and Bo Zhu for very helpful discussions, and Jiaping Wang and Ovidiu Munteanu for their interest in this work. This work was partially supported by the National Natural Science Foundation of China (Grant No. 12141101).

\section{Weighted $L^2$ cohomology }\label{section: weighted L2 cohomology}
This section collects the necessary definitions and fundamental properties of weighted $L^2$-cohomology, together with some useful calculations. Since the arguments are largely standard, we will omit some proofs and refer to \cite{Bue99}, \cite{Carron2002}, \cite{Schwarz95}, etc. We do give detailed proofs where the weight function plays a significant role.

\subsection{Definitions and basic properties}
Let $(M, g, dv_f)$ be a smooth metric measure space, where $(M, g)$ is a complete Riemannian manifold and $f \in C^\infty(M)$ is a smooth function called the potential function. Denote by $dv$ the Riemannian volume measure, and define the $f$-measure by $dv_f = e^{-f} dv$.

Let $L^2_f\Lambda^p(M)$ denote the space of $p$-forms on $M$ that are $L^2$-integrable with respect to the measure $dv_f$. Since the underlying Riemannian manifold is complete, elements of $L^2_f\Lambda^p(M)$ can be approximated by compactly supported smooth $p$-forms on $M$. The exterior derivative $d$ can be defined as a densely defined operator on $L^2_f\Lambda^p(M)$ in the distributional sense: for $\omega \in L^2_f\Lambda^p(M)$, if there exists $\eta \in L^2_f\Lambda^{p+1}(M)$ such that
\[
\int_M \omega \wedge d \varphi = - \int_M \eta \wedge \varphi , 
\]
for every compactly supported smooth $(n-p-1)$-form $\varphi$, then we say that $\omega$ lies in the domain of $d$ and that $d\omega = \eta$.

 Define 
 $$\Lambda_f^p(M) = \{\omega \in L^2_f\Lambda^p(M): d\omega \in L^2_f\Lambda^{p+1}(M)\},$$
which is the domain of the operator $d_p: \Lambda_f^p(M) \to \Lambda_f^{p+1}(M)$. For simplicity, we will omit the subscript $p$ when there is no confusion.
An important fact is that $d$ (or $d_p$) is a closed operator. This operator $d$ coincides with the closure of the exterior derivative acting on compactly supported smooth differential forms, in particular, $\overline{d \Lambda^p_f(M)} = \overline{d C_0^\infty \Lambda^p(M)}$ (see \cite[Lemma 1.5]{Carron2002}).

We have a cochain complex 
\[
\cdots \xrightarrow{d_{p-1}} \Lambda_f^p(M) \xrightarrow{d_{p}} \Lambda_f^{p+1}(M) \xrightarrow{d_{p+1}} \cdots.
\]
The \textit{unreduced $L^2(dv_f)$-cohomology} is defined as 
\[
H_{(2)}^p(M, dv_f) = {\operatorname{ker}(d_p)} \mathbin{/} {\operatorname{im}(d_{p-1})};
\]
the \textit{reduced $L^2(dv_f)$-cohomology} is defined as
\[
\bar{H}_{(2)}^p(M, dv_f) = {\operatorname{ker}(d_p)} \mathbin{/}{\overline{\operatorname{im}(d_{p-1})}},
\]
where ${\overline{\operatorname{im}(d_{p-1})}}$ denotes closure of the image of $d_{p-1}$. There is a natural surjective map 
\[
H_{(2)}^p(M, dv_f) \to \bar{H}_{(2)}^p(M, dv_f) ; \quad [\omega]_{(2)} \mapsto [\omega]_{\overline{(2)}},
\]
whose kernel is $\overline{\operatorname{im}(d_{p-1})} / \operatorname{im}(d_{p-1})$. Here and in what follows, $[\cdot]_{(2)}$ and $[\cdot]_{\overline{(2)}}$ denote the cohomology classes in $H^*_{(2)}$ and $\bar{H}^*_{(2)}$, respectively. 

The following lemma, due to de Rham \cite{deRham1973} (see the presentation in \cite[Lemma 1.11]{Carron2002}), can be easily generalized to the weighted setting.
\begin{lem}\label{de Rham Lemma}
Let $\omega \in L^2_f\Lambda^p(M)$ be a closed form such that $[\omega]_{\overline{(2)}} = 0$ in the reduced weighted $L^2$-cohomology $\bar{H}_{(2)}^p(M, dv_f)$, then $[\omega] = 0$ in the de Rham cohomology $H_{dR}^p(M) $. 
\end{lem}
\begin{proof}
Since $[\omega]_{\overline{(2)}} = 0$, there is a sequence of $(p-1)$-forms $\eta_k \in \Lambda_f^{p-1}(M)$, such that $d \eta_k \to \omega$ in $L^2(dv_f)$-norm, we can further take $\eta_k$ to be compactly supported smooth forms.  

Let $\alpha$ be any closed $(n-p)$-form with compact support $\Omega$. Then
$$
\|\omega - d\eta_k\|_{L^2(dv, \Omega)} \leq e^{\sup_\Omega f} \|\omega - d\eta_k\|_{L^2(dv_f, M)} \to 0,
$$ 
and consequently,
\[
\int_M \omega \wedge \alpha = \lim_{k \to \infty} \int_M d\eta_k \wedge \alpha = \lim_{k \to \infty} \int_M d(\eta_k \wedge \alpha) = 0. 
\]
By Poincar\'e duality, the pairing $H_{dR}^p(M) \times H_{0}^{n-p}(M) \to \mathbb{R}$ given by 
\[
([\alpha], [\beta]) \mapsto \int_M \alpha \wedge \beta
\]
is nondegenerate, where $H_{0}^{n-p-1}(M)$ denotes the compactly supported de Rham cohomology. This implies that $\omega$ is exact. 
\end{proof}

Denote by $\delta$ the $L^2$-adjoint of $d$ with respect to the Riemannian measure $dv$, and let $\delta_f$ be the $L^2$-adjoint of $d$ with respect to the weighted measure $dv_f = e^{-f}dv$. A direct computation shows that 
$$\delta_f = \delta + \iota_{\n f},$$
where $\iota_{\nabla f}$ denotes interior multiplication by the vector field $\nabla f$. It also holds that $\delta_f^2 = 0$, which can be verified by duality or by direct calculation \cite{Bue99}.
We have
$$
(\ker(d_p))^\perp = \overline{\operatorname{im}(\delta_f)_{p+1}}\quad \text{and} \quad (\ker(\delta_f)_p)^\perp = \overline{\operatorname{im}(d_{p-1})},
$$ 
where the orthogonal complements and closures are taken in the space of $L^2(dv_f)$-integrable forms. 

Let $\mathcal{H}_f^{p}(M)$ denote the space of $L^2(dv_f)$-harmonic $p$-forms on $M$, i.e. 
\[
\mathcal{H}_f^{p}(M) = \{\omega \in \Lambda_f^p(M): d \omega = \delta_f \omega = 0\}.
\]
This space is closed in $L^2_f\Lambda^p(M)$ since both $d$ and $\delta_f$ are closed operators. Elements of $\mathcal{H}_f^{p}(M)$ will be called $f$-harmonic $p$-forms on $M$.

Then we have the Hodge decomposition (\cite[Theorem 5.7]{Bue99})
\begin{equation}\label{eqn: Hodge decomposition}
L^2_f\Lambda^p(M) = \mathcal{H}_f^p(M) \oplus \overline{\operatorname{im}(d_{p-1})} \oplus \overline{\operatorname{im}(\delta_f)_{p+1}}.
\end{equation}
Consequently, there is a canonical isomorphism between the space of $f$-harmonic forms and the reduced weighted $L^2$-cohomology:
\begin{equation}\label{eqn: reduced L2 chomology and harmonic forms}
\mathcal{H}_f^p(M) \cong \bar{H}_{(2)}^p(M).
\end{equation}

Define the $f$-Hodge Laplacian operator to be 
\[
\Delta^d_f = d\delta_f + \delta_f d. 
\]
This operator admits a densely defined self-adjoint extension on $\Lambda_f^*(M)$, obtained as the operator associated to the positive semi-definite quadratic form 
$$Q(\alpha) = \int_M (|d\alpha|^2 + |\delta_f  \alpha|^2 )dv_f.$$ 
We denote by $(\Delta^d_f)_p$ its restriction to $p$-forms. Standard argument (\cite[Theorem 5.5]{Bue99}) shows that
$$\ker(\Delta^d_f)_p = \mathcal{H}_f^p.$$

\subsection{Weitzenb\"ock formula for the $f$-Hodge Laplacian.}
We now compute the curvature term appearing in the $f$-Hodge Laplacian. Let $e_1,\dots,e_n$ be an orthonormal frame, and let $\theta^1,\dots,\theta^n$ be the dual orthonormal coframe. A $p$-form $\omega$ can be expressed locally as
\[
\omega = \frac{1}{p !} \sum_{i_1, i_2,..., i_p}\omega_{i_1, i_2,..., i_p} \theta^{i_1} \wedge \theta^{i_2} \wedge \cdots \wedge \theta^{i_p}.
\]
Recall that for the usual (unweighted) Hodge Laplacian, the Weitzenb\"ock formula takes the form
\[
\Delta^d \omega= - \Delta \omega+  \mathcal{R}(\omega),
\]
where $\Delta = -\nabla^*\nabla$ is the rough Laplacian, and $\mathcal{R}$ is a linear operator acting on forms. In terms of the Riemann curvature tensor,
\begin{equation}\label{eqn: curvature term in Hodge Laplacian}
\mathcal{R}(\omega) = \theta^i \wedge \iota_{e_j} R(e_i, e_j) \omega,
\end{equation}
where the action of $R(X,Y)$ on a $p$-form is defined by
\[
\left( R(X, Y) \omega \right) (Z_1,..., Z_p) = \sum_{k = 1}^p \omega(Z_1,..., R(X, Y) Z_k, ..., Z_p).
\]
We adopt the curvature convention that $R_{ijkl} = g(R(e_i, e_j)e_l, e_k)$.

A direct computation shows that 
$$\Delta_f^d = \Delta^d + \mathcal{L}_{\n f},$$ 
where $\mathcal{L}_{\nabla f}$ denotes the Lie derivative in the direction of $\nabla f$. By Cartan's magic formula,
\[
\mathcal{L}_{\n f} \omega =  d i_{\n f} \omega + i_{\n f} d \omega .
\]
We now compute each term explicitly in a local orthonormal coframe. For a $p$-form $\omega$, we have
\[
\begin{split}
d i_{\n f} \omega = & d \left( \frac{1}{p!} \sum_{s=1}^p f_{i_s} \omega_{i_1...i_p} (-1)^{s-1} \theta^{i_1} \wedge \cdots \wedge \widehat{\theta^{i_s}} \wedge \cdots \wedge \theta^{i_p} \right) \\
= & \frac{1}{p!} \sum_{s=1}^p (f_{i_s j} \omega_{i_1...i_p} + f_{i_s}\n_j \omega_{i_1...i_p} ) (-1)^{s-1} \theta^j \wedge \theta^{i_1} \wedge \cdots \wedge \widehat{\theta^{i_s}} \wedge \cdots \wedge \theta^{i_p} \\
= & \frac{1}{p!} \sum_{s=1}^p \sum_j (f_{i_s j} \omega_{i_1...j...i_p} + f_{j}\n_{i_s} \omega_{i_1...j...i_p} )  \theta^{i_1} \wedge \cdots \wedge {\theta^{i_s}} \wedge \cdots \wedge \theta^{i_p}, \\
\end{split}
\]
where in the last line $j$ appears in the $s$-th slot and the hat denotes omission. Similarly, 
\[
\begin{split}
i_{\n f} d \omega = & \frac{1}{p!} \sum_j f_j \n_j \omega_{i_1...i_p} \theta^{i_1} \wedge \cdots \wedge \theta^{i_p} + \frac{1}{p!} \sum_{s=1}^p f_{i_s} \n_j \omega_{i_1...i_p} (-1)^s \theta^j \wedge \theta^{i_1} \wedge \cdots \wedge \widehat{\theta^{i_s}} \wedge \cdots \wedge \theta^{i_p} \\
= & \n_{\n f} \omega - \frac{1}{p!} \sum_{s=1}^p f_{i_s} \n_j \omega_{i_1...i_p}  \theta^{i_1} \wedge \cdots \wedge {\theta^{j}} \wedge \cdots \wedge \theta^{i_p} \\
= & \n_{\n f} \omega - \frac{1}{p!} \sum_{s=1}^p \sum_j  f_{j} \n_{i_s} \omega_{i_1...j...i_p}  \theta^{i_1} \wedge \cdots \wedge {\theta^{i_s}} \wedge \cdots \wedge \theta^{i_p}, \\
\end{split}
\]
where in the last line $j$ occupies the $s$-th position. Adding the two contributions, the terms involving derivatives of the coefficients cancel, yielding
\[
\mathcal{L}_{\n f} \omega = \n_{\n_f} \omega + \frac{1}{p!} \sum_{s=1}^p \sum_j f_{i_s j} \omega_{i_1...j...i_p} \theta^{i_1} \wedge \cdots \wedge \theta^{i_p}.
\]

Combining this with the standard Weitzenb\"ock formula $\Delta^d \omega = -\Delta \omega + \mathcal{R}(\omega)$, we obtain the weighted Weitzenb\"ock formula
\begin{equation}\label{eqn: weighted Weitzenbock formula}
\Delta_f^d \omega  = \Delta^d \omega + \n_{\n f} \omega + \mathcal{F}(\omega)= - \Delta \omega + \n_{\n f} \omega + \mathcal{R}_f(\omega),
\end{equation}
where 
\begin{equation}\label{eqn: R_f}
\mathcal{R}_f(\omega) = \mathcal{R}(\omega) + \mathcal{F}(\omega).
\end{equation}
and $\mathcal{F}(\omega)$ is the curvature-type term arising from the Hessian of $f$:
\begin{equation}\label{eqn: formula for F}
\mathcal{F}(\omega) =  \frac{1}{p!} \sum_{s=1}^p \sum_{j, i_1, ..., i_p = 1}^n f_{i_s j} \omega_{i_1...j...i_p} \theta^{i_1} \wedge \cdots \wedge \theta^{i_p},
\end{equation}
where $f_{ij} = \nabla_i \nabla_j f$ denotes the Hessian of $f$, and in the term $\omega_{i_1\cdots j\cdots i_p}$ the index $j$ appears in the $s$-th position.

We now derive an alternative expression for the curvature term $\mathcal{R}_f$.
\begin{lem} \label{lem: decomposing the curvature term into Ric_f and curvature operator}
Let $\operatorname{Ric}_f = \operatorname{Ric} + \n\n f$, then we can write
\[
\mathcal{R}_f(\omega) = \frac{1}{p!} \left( \sum_{s=1}^p\sum_{j=1}^n (\operatorname{Ric}_f)_{i_s j}  \omega_{ i_1 ... j ... i_p}  - \frac{1}{2}\sum_{t\neq s = 1}^p \sum_{k,l=1}^n R_{ i_t i_s l k } \omega_{i_1 ...  l  ...  k ... i_p}  \right) \theta^{i_1}\wedge \cdots \wedge \theta^{i_p}
\]
where in the first term $j$ appears in the $s$-th slot, and in the second term $l$ and $k$ appear in the $t$-th and $s$-th slots respectively. The indices $i_1,\dots,i_p$ run from $1$ to $n$.
\end{lem}
\begin{proof}
Recall from the standard Weitzenb\"ock formula that the curvature term $\mathcal{R}(\omega)$ can be expressed as
\[
\begin{split}
\mathcal{R}(\omega) = & \frac{1}{p!} \sum  R_{i_t l i_s k} \omega_{ i_1 ... {i_{s-1}} k {i_{s+1}}... i_p}   \theta^{i_1}\wedge \cdots\wedge \theta^{i_{t-1}} \wedge \theta^l \wedge \theta^{i_{t+1}} \wedge \cdots \wedge \theta^{i_p} \\
= & \frac{1}{p!} \left( \sum R_{i_s k } \omega_{ i_1 ... {i_{s-1}} k {i_{s+1}}... i_p}  + \sum_{t\neq s} R_{l i_t i_s k } \omega_{i_1 ... i_{t-1} l i_{t+1} ... i_{s-1} k i_{s+1} ... i_p}  \right) \theta^{i_1}\wedge \cdots \wedge \theta^{i_p}, \\
\end{split}
\]
where $R_{i_s k}$ denotes the Ricci curvature components.
For the off-diagonal terms with $t \neq s$, splitting the sum into $t < s$ and $t > s$, and relabeling indices, we obtain
\[
\begin{split}
& \sum_{t\neq s} R_{l i_t i_s k } \omega_{i_1 ... i_{t-1} l i_{t+1} ... i_{s-1} k i_{s+1} ... i_p} \\
= & \sum_{t< s} R_{l i_t i_s k } \omega_{i_1 ... i_{t-1} l i_{t+1} ... i_{s-1} k i_{s+1} ... i_p} + \sum_{t > s} R_{l i_t i_s k } \omega_{i_1 ... i_{s-1} k i_{s+1} ... i_{t-1} l i_{t+1} ... i_p} \\
= & \sum_{t< s} R_{l i_t i_s k } \omega_{i_1 ... i_{t-1} l i_{t+1} ... i_{s-1} k i_{s+1} ... i_p} + \sum_{t < s} R_{l i_s i_t k } \omega_{i_1 ... i_{t-1} k i_{t+1} ... i_{s-1} l i_{s+1} ... i_p} \\
= & \sum_{t< s} (R_{l i_t i_s k } - R_{l i_s i_t k})\omega_{i_1 ... i_{t-1} l i_{t+1} ... i_{s-1} k i_{s+1} ... i_p} \\
= & - \sum_{t< s} R_{l k i_t i_s}\omega_{i_1 ... i_{t-1} l i_{t+1} ... i_{s-1} k i_{s+1} ... i_p},
\end{split}
\]
where the last equality follows from the first Bianchi identity together with the symmetry $-R_{l i_s i_t k} = R_{l i_s k  i_t}$. After symmetrizing over $t$ and $s$, we arrive at 
\[
\mathcal{R}(\omega)= \frac{1}{p!} \left( \sum R_{i_s k } \omega_{ i_1 ... {i_{s-1}} k {i_{s+1}}... i_p}  - \frac{1}{2}\sum_{t\neq s} R_{ i_t i_s l k } \omega_{i_1 ... i_{t-1} l i_{t+1} ... i_{s-1} k i_{s+1} ... i_p}  \right) \theta^{i_1}\wedge \cdots \wedge \theta^{i_p} \\
\]
where in the second sum $l$ and $k$ occupy the $t$-th and $s$-th positions, respectively; the indices $i_1,...,i_p, k, l$ range from $1$ to $n$, and the indices $s, t$ range from $1$ to $p$.

Finally, combining this expression with the formula for $\mathcal{F}(\omega)$ from (\ref{eqn: formula for F}) and using the definition $\operatorname{Ric}_f = \operatorname{Ric} + \nabla\nabla f$, we obtain the desired decomposition.
\end{proof}

\subsection{Spectrum of the $f$-Hodge Laplacian. }

Consider the operator defined on $L^2(dv)$-forms:
\begin{equation}\label{eqn: definition of the operator L}
L = e^{-f/2} \Delta^d_f e^{f/2} = -\Delta + V_f + \mathcal{R}_f, 
\end{equation}
where 
\begin{equation}\label{eqn: definition of V_f}
V_f = \frac{1}{4} |\n f|^2 - \frac{1}{2} \Delta f. 
\end{equation}  
Clearly $\Delta_f^d$ and $L$ are isospectral, with the unitary equivalence given by the multiplication operator $e^{-f/2}: L^2(M, dv_f) \to L^2(M, dv)$. Here, for simplicity, we write $L^2(\cdot)$ to denote the space of $L^2$-integrable forms with respect to the indicated measure, without specifying the degree of the forms.

Since $\Delta_f^d$ is a nonnegative operator, $L$ is also nonnegative. Moreover, $L$ is a Schr\"odinger-type operator. By well-known spectral theory (see \cite[page 120]{RS78}), $L$ has discrete spectrum if the potential (understood here as an operator) grows sufficiently fast — specifically, if its lower bound tends to positive infinity at geometric infinity.
\begin{lem}\label{lem: discrete spectrum}
Let $\lambda: M \to \mathbb{R}$ be a continuous function such that for every $p$-form $\alpha$,
\[
V_f(x)|\alpha|^2(x) + \langle \mathcal{R}_f(\alpha) , \alpha \rangle (x) \geq \lambda (x) |\alpha|^2(x), \quad \forall x\in M, 
\] 
and assume $\lambda(x) \to +\infty$ as $x \to \infty$. Choose a constant $\bar{\lambda}$ such that $\bar{\lambda} + \lambda(x) > 0$ for all $x$ (possible because $\lambda$ is bounded below on compact sets and diverges to $+\infty$). Then the operator $(\bar{\lambda} + L)^{-1}$ is compact; consequently, $L$ has purely discrete spectrum.
\end{lem}
\begin{proof}
The proof follows standard argument, we include it here for readers' convenience. 

For $\alpha \in W^{1,2}(M, dv)$, define the quadratic form
\[
Q(\alpha) : = \int_M \left( |\n \alpha|^2 + V_f|\alpha|^2 + \langle \mathcal{R}_f(\alpha), \alpha\rangle \right) dv. 
\]
Let $\{\beta_n\}_{n=1}^\infty$ be a sequence of $p$-forms with $\|\beta_n\|_{L^2(M,dv)} = 1$, and let $\alpha_n$ satisfy $(\bar{\lambda} + L)\alpha_n = \beta_n$, $n = 1,2,\dots$. We will show that $\{\alpha_n \}$ is precompact in $L^2(M,dv)$.

From the definition of $L$ and integration by parts,
\[
\bar{\lambda} \|\alpha_n\|_{L^2(M, dv)}^2 + Q(\alpha_n) = (\alpha_n, \beta_n) \leq \epsilon \|\alpha_n\|_{L^2(M, dv)}^2 + \epsilon^{-1}\|\beta_n\|_{L^2(M, dv)}^2,   
\]
for any $\epsilon > 0$. Using the assumption $V_f|\alpha|^2 + \langle \mathcal{R}_f(\alpha),\alpha\rangle \ge \lambda |\alpha|^2$, we obtain  
\[
\bar{\lambda} \|\alpha_n\|_{L^2(M, dv)}^2 + Q(\alpha_n) \geq \inf_M(\bar{\lambda} +  \lambda(x)) \int_M |\alpha_n|^2 dv + \int_M |\n \alpha_n|^2 dv.
\]
Since $\lambda(x) \to \infty$ as $x\to \infty$, there is a positive constant $\epsilon_0$ such that $\bar{\lambda}+\lambda>\epsilon_0$ uniformly. Take $\epsilon < \epsilon_0$, the above inequalities yield uniform $W^{1,2}$-bounds on ${\alpha_n}$ and uniform bounds on $Q(\alpha_n)$,
\[
 \|\alpha_n\|_{W^{1,2}(M, dv)} \leq C \quad and \quad  Q(\alpha_n) \leq C, \quad \forall n. 
\]
Let $M_R = \{x \in M: \lambda(x) \leq R\}$. By $\lambda(x) \to +\infty$, each $M_R$ is compact for all $R$. On $M_R$, $\lambda$ is bounded below, so we can write
\[
Q(\alpha_n) \geq - \sup_{x \in M_R}\lambda_-(x)\int_{M_R} |\alpha_n|^2 dv+ R \int_{M \setminus M_R} |\alpha_n|^2 dv,
\]
where $\lambda_-(x)$ is a bounded compactly supported function. 

Rearranging the above inequality, we observe that for any $\epsilon > 0$, we can choose $R$ sufficiently large such that 
\begin{equation}\label{eqn: concentration of L2 integral}
 \int_{M \setminus M_R} |\alpha_n|^2 dv \leq \frac{Q(\alpha_n)}{R} + \frac{\sup_{x \in M_R}\lambda_-(x)}{R} \|\alpha_n\|_{L^2(M, dv)} < \epsilon, \quad \forall n. 
\end{equation}
Take a sequence $\epsilon_i \to 0$ and corresponding $R_i \to \infty$. Since $\lambda(x)$ is merely continuous, the sublevel sets $M_{R_i}$ may not have smooth boundary. To remedy this, we can approximate $\lambda$ by a smooth function $\tilde{\lambda}$, and choose $R_i$ to be regular values of $\tilde{\lambda}$. Then the sets $\widetilde{M}_{R_i} = \{x : \tilde{\lambda}(x) \le R_i\}$ are compact domains with smooth boundary, and they exhaust $M$ as $i \to \infty$. So, without loss of generality we can assume $M_{R_i}$ to have smooth boundary. By Rellich theorem $W^{1,2}(M_{R_i},dv) \hookrightarrow L^2(M_{R_i}, dv)$ is compact for each $R_i$, hence we can extract a diagonal subsequence $\alpha_{n_i}$ which is convergent in each $L^2(M_{R_i})$. The uniform tail estimate above then implies ${\alpha_{n_i}}$ is Cauchy in $L^2(M)$. Thus ${(\bar{\lambda}+L)^{-1}\beta_n}$ has a convergent subsequence, proving compactness of the resolvent and hence discreteness of the spectrum of $L$.
\end{proof}

Consequently we can prove the following:
\begin{lem}\label{lem: identification of reduced and unreduced L2 cohomology} Under the assumption of Lemma \ref{lem: discrete spectrum}, the image of $d_{p-1}$ is closed. Consequently, $H_{(2)}^p(M, dv_f)$ is isomorphic to $\bar{H}_{(2)}^p(M, dv_f)$.
\end{lem}

\begin{proof}
Since $\Delta_f^d$ is nonnegative and has discrete spectrum, there exists a constant $c > 0$ such that for any $\alpha \in (\ker d_{p-1})^\perp$,
\[
c \|\alpha\|^2_{L^2(dv_f)} \leq  (\Delta_f^d \alpha, \alpha)  = \|d\alpha\|^2_{L^2(dv_f)},
\]
where we have used that $(\ker d_{p-1})^\perp = \overline{\operatorname{im}((\delta_f)_p)}$, and on this subspace $\Delta_f^d$ restricts to $\delta_f d$.

Now suppose $\{\alpha_n\} \subset L^2_f\Lambda^{p-1}(M)$ is a sequence such that $d\alpha_n \to \beta$ in $L^2_f\Lambda^p(M)$, we need to prove that $\beta \in \operatorname{im}(d_{p-1})$. We can take $\alpha_n'$ to be the projection of $\alpha_n$ onto $(\ker d_{p-1})^\perp$, which satisfies $d\alpha'_n = d\alpha_n$. Then
\[
\|\alpha'_n - \alpha'_m\|^2_{L^2(dv_f)} \leq c^{-1} \|d\alpha'_n - d\alpha'_m\|^2_{L^2(dv_f)}
\]
implies that ${\alpha'_n}$ is Cauchy. Hence $\alpha' = \lim_{n\to\infty} \alpha'_n$ exists in $L^2_f\Lambda^p(M)$. By the closedness of $d$, we have $d\alpha' = \beta$, so $\beta \in \operatorname{im}(d_{p-1})$. Thus $\operatorname{im}(d_{p-1})$ is closed.
\end{proof}

If we have a positive Bakry-Emery lower bound $\operatorname{Ric}_f \geq a g$ for some $a> 0$, there is an alternative approach to proving discreteness of the spectrum. Indeed, by the argument of Hein-Naber \cite{HN14} and Cheng-Zhou \cite{CZ17}, we can show that $W^{1,2}(M, dv_f)$ is compactly embedded in $L^2(M, dv_f)$. The same method yields that
\[
W^{1,2}(\Lambda^p, dv_f) = \{\alpha \in \Lambda^p(M): \int_M |\alpha|^2 + |\n \alpha|^2 dv_f\}
\]
is compactly embedded in the space of $L^2(dv_f)$-forms. 

However, this alone is not sufficient to show the discreteness of the spectrum of $(\Delta^d_f)_p$ for general $p$. The reason is that the space of forms with finite energy,
\[
\{\alpha \in \Lambda^p(M): \int_M |\alpha|^2 + |d \alpha|^2 + |\delta_f \alpha|^2 dv_f< \infty\}
\]
may be larger than $W^{1,2}(\Lambda^p, dv_f)$. Recall that the Weitzenb\"ock formula gives
\[
\int_M |\alpha|^2 + |d \alpha|^2 + |\delta_f \alpha|^2 dv_f = \int_M |\alpha|^2 + |\n \alpha|^2 + \langle \mathcal{R}_f(\alpha), \alpha \rangle dv_f.
\]
Therefore, a uniform lower bound on $\mathcal{R}_f$, in addition to the positive Bakry–Emery lower bound, suffices to ensure that $\Delta_f^d$ has purely discrete spectrum. This extra condition guarantees that the energy form associated to $\Delta_f^d$ dominates the $W^{1,2}$-norm, allowing the compact embedding to take effect.

\subsection{$f$-harmonic forms on bounded domains}

We will also need the theory of weighted harmonic forms on compact domains. In the sequel, we let $\Omega \subset M$ be a compact sublevel set of $f$ with smooth boundary, i.e., $\Omega = \{ f \leq R \}$ for some regular value $R$ of $f$. In the unweighted setting, the space of $L^2$-harmonic forms satisfying the Neumann (absolute) boundary condition is isomorphic to the absolute de Rham cohomology of $\Omega$ (see, e.g., \cite[Chapter~2]{Schwarz95}). The weighted case is entirely analogous; we present the necessary modifications below.

For a differential form $\alpha$ on $\Omega$, let $i:\partial\Omega\hookrightarrow\Omega$ be the inclusion, then $\bm{t} \alpha = i^* \alpha$ is the tangential part of $\alpha$; let $\bm{n } \alpha = \alpha - \bm{t} \alpha $ be the normal part. It is evident that $d\bm{t} = \bm{t} d$, and there is a relation $* \bm{n} = \bm{t} *$ since the Hodge star swaps tangential and normal components. Then by $*\delta =  (-1)^p d *$ (where $p$ is the degree of the form on which it acts) we can derive that $\bm{n} \delta = \delta \bm{n}$. See \cite[Prop.~1.2.6]{Schwarz95} for detailed proofs of the properties of $\bm{t}$ and $\bm{n}$.

Recall the Green's formula for differential forms \cite[Prop.~2.1.2]{Schwarz95}, 
\begin{equation}\label{eqn: Green's formula}
\int_\Omega \langle \alpha, d \beta \rangle dv - \int_\Omega \langle \delta \alpha,  \beta \rangle dv = \int_{\pd \Omega} \bm{t} \beta \wedge * \bm{n} \alpha \hspace{2pt} .
\end{equation}
Apply it to $e^{-f} \alpha$ and $\beta$, note that
\[
\delta(e^{-f}\alpha) = -\sum_i \iota_{e_i} \n_{e_i} (e^{-f}\alpha) = e^{-f} \iota_{\n f} \alpha + e^{-f }\delta \alpha = e^{-f} \delta_f \alpha,
\]
and that $*\bm{n}(e^{-f} \alpha) = e^{-f} * \bm{n} \alpha$, hence we get the following weighted version of Green's formula
adapted to the measure $dv_f$:
\[
\int_\Omega \langle \alpha, d \beta \rangle dv_f - \int_\Omega \langle \delta_f \alpha,  \beta \rangle dv_f = \int_{\pd \Omega} \bm{t} \beta \wedge * \bm{n} \alpha \hspace{2pt} e^{-f}.
\]

We first address the adjustments required for the boundary conditions. 
In view of the weighted Green's formula, the natural boundary conditions for the weighted Hodge Laplacian $\Delta_f^d$ are as follows:
\begin{itemize}
    \item {Dirichlet (relative) boundary condition:} $\bm{t}\alpha = \bm{t}(\delta_f \alpha) = 0$;
    \item {Neumann (absolute) boundary condition:} $\bm{n} \alpha = \bm{n} (d\alpha) = 0$.
\end{itemize} 
Note that the Neumann boundary condition is exactly the same as in the unweighted case, while in the Dirichlet condition $\delta$ is replaced by $\delta_f$. For $f$-harmonic forms these reduce to $\bm{t}\alpha = 0$ and $\bm{n} \alpha = 0$, respectively.

To introduce the Hodge-Morrey decomposition in the weighted case, let 
\[
\Lambda_{D}^p(\Omega) = \{\alpha \in \Lambda^p(\Omega): \bm{t} \alpha = 0\}, \quad and \quad \Lambda_{N}^p(\Omega) = \{\alpha \in \Lambda^p(\Omega): \bm{n} \alpha = 0\}.
\]
From the Green's formula we obtain 
\[
(\ker(d_p))^\perp = \overline{\delta_f(\Lambda^{p+1}_{N}(\Omega) ) } \quad \text{and} \quad (\ker(\delta_f)_p)^\perp = \overline{d(\Lambda^{p-1}_{D}(\Omega) ) },
\] 
where the orthogonal complement and the closures are taken in the space of $L^2(dv)$-forms, or in the space of $L^2(dv_f)$-forms, which are equivalent on compact domains where $f$ is smooth and bounded. Since $\Omega$ is compact, $\Delta_f^d$ with either Dirichlet or Neumann boundary condition has discrete spectrum, hence by the same proof of Lemma \ref{lem: identification of reduced and unreduced L2 cohomology}, we can show that $d(\Lambda^{p-1}_{D}(\Omega) ) $ is closed; similarly we can show that $\delta_f(\Lambda^{p+1}_{N}(\Omega) )$ is also closed.

Take the orthogonal complement in the space of $L^2(dv_f)$-forms
\[
\mathcal{H}_f^p(\Omega): = \left( d(\Lambda^{p-1}_{D}(\Omega) ) \oplus \delta_f(\Lambda^{p+1}_{N}(\Omega) ) \right)^\perp. 
\]
This leads to the Hodge-Morrey decomposition 
\begin{equation}\label{eqn: Hodge-Morrey decomposition}
L^2_f\Lambda^p(\Omega) = \mathcal{H}_f^p (\Omega) \oplus  {d(\Lambda^{p-1}_{D}(\Omega) ) } \oplus  {\delta_f(\Lambda^{p+1}_{N}(\Omega) ) }.
\end{equation}
The space $\mathcal{H}_f^p (\Omega)$ consists of $L^2(dv_f)$-integrable $f$-harmonic forms. 

We also have the Friedrich decomposition
\begin{equation}\label{eqn: Friedrich decomposition}
\mathcal{H}_f^p(\Omega) = \mathcal{H}_{f, N}^p(\Omega) \oplus  \mathcal{H}_{f, ex}^p(\Omega),
\end{equation}
where $\mathcal{H}_{f, N}^p(\Omega)$ is the space of $f$-harmonic forms with Neumann boundary condition, which is finite dimensional since $\Omega$ is compact, and $\mathcal{H}_{f, ex}^p(\Omega)$ is the space of exact $f$-harmonic forms. \eqref{eqn: Friedrich decomposition} follows from \cite[Theorem~2.4.8]{Schwarz95} with minor modification: For any $\alpha \in \mathcal{H}_f^p(\Omega) \cap \mathcal{H}_{f, N}^p(\Omega)^\perp$, solve the Poisson equation with Neumann boundary condition
\[
\Delta_f^d \beta = \alpha \quad \text{ in }\quad \Omega; \quad \bm{n}\beta = \bm{n} d\beta = 0 \quad \text{on }\quad \partial \Omega. 
\]
Observe that the form $\alpha - d \delta_f \beta = \delta_f d \beta$ also lies in $\mathcal{H}_{f, N}^p(\Omega)^\perp$ by the Green's formula. It is $f$-harmonic since
\[
\Delta_f^d \delta_f d \beta = \delta_f d \Delta_f^d \beta = 0.
\]
And we can verify that it satisfies the Neumann boundary condition: note that the interior product $\iota_{\n f}$ does not create any normal component of a form, the  boundary condition $\bm{n} d \beta = 0$ means that $d\beta$ has vanishing normal component, hence $\iota_{\n f} d\beta$ must also have vanishing normal component, i.e. $\bm{n}(\iota_{\n f} d \beta) = 0$, then using the commutativity $\delta \bm{n} = \bm{n} \delta$, we obtain that
\[
\bm{n} \delta_f d \beta =\bm{n}(\delta d \beta + \iota_{\n f} d \beta)= \delta \bm{n} (d\beta) + \bm{n} (\iota_{\n f} d\beta) = 0.
\]
Consequently $\delta_f d \beta$ lies in $\mathcal{H}_{f, N}^p(\Omega)\cap \mathcal{H}_{f, N}^p(\Omega)^\perp$, hence is $0$. Then $\alpha = d\delta_f \beta$ is an exact $f$-harmonic form. This shows that $\mathcal{H}_f^p(\Omega) \cap \mathcal{H}_{f, N}^p(\Omega)^\perp \subset \mathcal{H}_{f, ex}^p(\Omega)$ and leads to \eqref{eqn: Friedrich decomposition}. 

Therefore, by \eqref{eqn: Hodge-Morrey decomposition} and \eqref{eqn: Friedrich decomposition} we have the Hodge isomorphism:
\begin{equation}\label{eqn: hodge thm abs}
H^p_{dR}(\Omega) \cong \mathcal{H}_{f, N}^p(\Omega),
\end{equation}
where $H^p_{dR}(\Omega)$ is the (absolute) de Rham cohomology group of $\Omega$.

\section{Agmon-type estimates and finite dimensionality for weighted harmonic forms}\label{section: growth estimates}

In this section, we prove decay estimates for differential forms $\omega$ satisfying the equation 
\[
L  \omega = ( -\Delta + V_f + \mathcal{R}_f )\omega = 0,
\]
or equivalently $e^{f/2}\omega$ is $f$-harmonic \eqref{eqn: definition of the operator L}. 
Our estimates are obtained in integral form, which can be turned into pointwise estimate by using mean value inequalities. 

The method is an adaptation of Agmon’s exponential decay estimate \cite{Agmon82}. Rather than using Agmon’s distance, we use the potential function $f$ directly. By compensating the exponential weight function with a polynomial decay term, we obtain the sharp exponential order.

By writing the operator $L$ in different forms, we prove two main lemmas under different sets of assumptions. Lemma \ref{lem: integral growth estimate} requires a lower bound of the curvature term in the Weitzenb\"ock formula; while Lemma \ref{lem: integral growth estimate - abs boundary condition} requires only assumptions on $f$.
\begin{lem}\label{lem: integral growth estimate}
Let $(M, g, dv_f= e^{-f}dv)$ be a smooth metric measure space, suppose there are  constants $a_1, \tilde{a} > 0$ and $b_1, \tilde{b} \geq 0$, such that 
\begin{enumerate}
\item $f \geq 0 $ and $f(x) \to \infty$ as $x\to \infty$;
\item $ |\n f|^2 \leq a_1 f +b_1$;
\item $V_f |\alpha|^2+ \langle \mathcal{R}_f(\alpha), \alpha\rangle  \geq (\tilde{a} f - \tilde{b})|\alpha|^2$ for every $p$-form $\alpha$,
\end{enumerate}
where $V_f$ is defined in \eqref{eqn: definition of V_f}. Let $\omega$ be a $p$-form in $L^2(dv)$ such that $L \omega = 0$, where $L$ is the Schr\"odinger type operator defined in \eqref{eqn: definition of the operator L}. Then for any $\sigma > 0$, and 
\[
q \geq \frac{ \sigma + \tilde{b} + a_1^{-1}\tilde{a} b_1 }{2\sqrt{a_1 \tilde{a}}} , 
\]
there are constants $C_1(q, \sigma, a_1, b_1, \tilde{a}, \tilde{b})$ and $C_2(q, a_1, b_1, \tilde{a}, \tilde{b}) $, such that 
\[
\sigma \int_{f> C_1}  (1+ f)^{-2q} e^{2\sqrt{\tilde{a}/a_1}f}|\omega|^2 dv \leq C_2\int_{f\leq C_1} |\omega|^2 dv.
\]
\end{lem}

\begin{proof}
First we show that $\omega \in W^{1,2}(dv)$. Let $\psi$ be a smooth cutoff function supported in a geodesic ball $B(p,2r)$ with $\psi \equiv 1$ on $B(p,r)$ and $|\nabla \psi| \le 2/r$. Integrating by parts and using $L\omega = 0$, we obtain
\[
\begin{split}
0 = & \int_M \psi^2 \langle\omega, -\Delta \omega + V_f \omega + \mathcal{R}_f(\omega) \rangle \\
= & \int_M \psi^2 |\n \omega|^2 + 2 \psi  \langle \n \psi \otimes \omega, \n \omega \rangle + \psi^2 V_f |\omega|^2 + \psi^2 \langle \mathcal{R}_f(\omega), \omega \rangle \\
\geq & \int_M \frac{1}{2} \psi^2 |\n \omega|^2 + (V_f \psi^2- 2|\n \psi|^2) |\omega|^2 + \psi^2  \langle \mathcal{R}_f(\omega), \omega \rangle, \\
\end{split}
\]
where we used Cauchy-Schwarz inequality to get the last inequality. By the assumptions, there exists a compact set $\Omega$ such that $V_f + \mathcal{R}_f \ge 1$ on $M\setminus\Omega$. Rearranging the above inequality gives
\[
\int_M \frac{1}{2} \psi^2 |\n \omega|^2 - 2|\n \psi|^2 |\omega|^2+ \int_{M \setminus \Omega} \psi^2 |\omega|^2  \leq - \int_\Omega V_f \psi^2 |\omega|^2 + \psi^2  \langle \mathcal{R}_f(\omega), \omega \rangle.
\]
Let $r \to \infty$, since $\omega$ is $L^2(dv)$, the term 
\[
\int_M |\n \psi|^2 |\omega|^2 dv \to 0.
\]
Hence we have 
\[
\int_{M} |\n \omega|^2 + \int_{M \setminus \Omega} |\omega|^2 \leq C \int_\Omega |\omega|^2
\] 
for some constant $C$ depending on $V_f$ and $\mathcal{R}_f$ on $\Omega$. Therefore $\omega$ is in $W^{1,2}(dv)$. 

Now fix a small number $\epsilon> 0$. For $\theta \in (0,\epsilon)$, define 
\[
f_\theta = \frac{f}{1+\theta f},
\]
which is a bounded positive function. The gradient $\n f_\theta = \frac{\n f}{(1+\theta f)^2}$ is also bounded on $M$ because $|\n f|^2 \leq a_1 f + b_1$. 
For parameters $q\geq 0$ and $\tau> 0$ to be chosen later, set
\[
\phi(s) = (1+s)^{-q} e^{\tau s}, \quad for \quad s \geq 0,
\]
and define $\omega_\theta = \phi(f_\theta) \omega$. Since $f_\theta$ is nonnegative and bounded, $|\phi(f_\theta)|$, $|\phi(f_\theta)^{-1}|$ and $|\n \phi(f_\theta)|$ are all bounded on $M$, the form $\omega_\theta$ belongs to $W^{1,2}(dv)$ as well. Consequently, integration by parts in the following calculation can be verified by a standard cutoff argument. 

Consider a family of conjugated operators $L_\theta = \phi(f_\theta) L (\phi(f_\theta))^{-1}$, then 
\[
L_\theta \omega_\theta = \phi(f_\theta) L \omega = 0.
\]
Taking inner product with $\omega_\theta$ and integrating by parts yields
\[
\begin{split}
0 = & \int_M \langle \phi(f_\theta) \omega_\theta, (- \Delta  + V_f  + \mathcal{R}_f )(\phi(f_\theta)^{-1} \omega_\theta)\rangle \\
= & \int_M \langle \n (\phi(f_\theta) \omega_\theta) , \n (\phi(f_\theta)^{-1} \omega_\theta) \rangle + V_f |\omega_\theta|^2 +\langle  \mathcal{R}_f(\omega_\theta), \omega_\theta \rangle \\
= &  \int_M ( V_f - |\n \ln \phi(f_\theta)|^2 ) |\omega_\theta|^2 +\langle  \mathcal{R}_f(\omega_\theta), \omega_\theta \rangle + |\n \omega_\theta|^2 \\
\geq & \int_M(\tilde{a} f- \tilde{b} - |\n \ln \phi(f_\theta)|^2) |\omega_\theta|^2.
\end{split}
\]
By direct calculation, 
\[
|\n \ln \phi(f_\theta)|^2  = \left( \tau -\frac{q}{1+f_\theta} \right)^2 \frac{|\n f|^2}{(1+\theta f)^4} \leq \left( \tau -\frac{q}{1+f_\theta} \right)^2 (a_1 f + b_1),
\]
where we used $|\nabla f|^2 \le a_1 f + b_1$ and $(1+\theta f)^{-4}\le 1$. Eventually we will let $\theta \to 0$.

Now fix $\tau = \sqrt{\tilde{a}/a_1}$, then we have 
\[
\begin{split}
0 \geq & \int_M \left[ \left( \tilde{a}  - a_1 \left( \sqrt{\frac{\tilde{a}}{a_1}}- \frac{q}{1+f_\theta} \right)^2\right) f - \tilde{b} - \left( \sqrt{\frac{\tilde{a}}{a_1}}  -\frac{q}{1+f_\theta} \right)^2 b_1 \right] |\omega_\theta|^2 \\
& = \int_M \left[ \frac{a_1q}{1+f_\theta}\left( 2 \sqrt{\frac{\tilde{a}}{a_1}}- \frac{q}{1+f_\theta} \right) f - \tilde{b} - \left( \sqrt{\frac{\tilde{a}}{a_1}}  -\frac{q}{1+f_\theta} \right)^2 b_1 \right] |\omega_\theta|^2.
\end{split}
\]
Observe that $\frac{f}{1+f_\theta} \geq 1$, and $f_\theta \to \infty$ as $\theta \to 0$ and $f \to \infty$, hence for fixed $q$, the term in the bracket above has limit infimum no less than $2q\sqrt{a_1\tilde{a}} - \tilde{b} -\frac{\tilde{a} b_1}{a_1}$ as $\theta \to 0$ and $f \to \infty$. 
For any $\sigma > 0$, take 
\[
q \geq \frac{ \sigma + \tilde{b} + a_1^{-1}\tilde{a} b_1 }{2\sqrt{a_1 \tilde{a}}} , 
\]
then there is a constant $C_1 = C(q, \sigma, a_1, b_1, \tilde{a}, \tilde{b})$, such that 
\[
 \left[ \frac{a_1q}{1+f_\theta}\left( 2 \sqrt{\frac{\tilde{a}}{a_1}}- \frac{q}{1+f_\theta} \right) f - \tilde{b} - \left( \sqrt{\frac{\tilde{a}}{a_1}}  -\frac{q}{1+f_\theta} \right)^2 b_1 \right]\geq \sigma
\]
when $f > C_1$, for all $\theta > 0$ sufficiently small. Consequently,
\[
\sigma \int_{f> C_1} |\omega_\theta|^2 dv \leq C(q, a_1, b_1, \tilde{a}, \tilde{b}) \int_{f\leq C_1} |\omega_\theta|^2 dv
\]
for all $\theta > 0$ sufficiently small. Then we can let $\theta \to 0$ and derive that 
\[
\sigma \int_{f> C_1} (1+ f)^{-2q} e^{2\sqrt{\tilde{a}/a_1}f}|\omega|^2 dv \leq C(q, a_1, b_1, \tilde{a}, \tilde{b}) \int_{f\leq C_1} |\omega|^2 dv,
\]
this completes the proof.
\end{proof}

\begin{remark} We make a few remarks about Lemma \ref{lem: integral growth estimate}.
\begin{enumerate}
\item The method clearly also works for eigenfunctions (or sections) of Shr\"odinger type equations $(\Delta  - V) u = -\lambda u$, provided the potential function $V$ satisfies appropriate growth conditions. The resulting exponential rate will depend on the eigenvalue $\lambda$.
\item Consider the Ornstein-Uhlenbeck operator $-\Delta +  \vec{x} \cdot \n$ on $\mathbb{R}^n$, it's corresponding isospectral Shrodinger type operator is $L = -\Delta + (\frac{1}{4} |\vec{x}|^2 - \frac{n}{2})$. The eigenfunctions of $L$ are Hermite polynomials multiplied by $e^{-|\vec{x}|^2/4}$.In this case, taking $f = \frac{1}{2}|\vec{x}|^2$, we can choose $a_1 = 2$ and $\tilde{a} = \frac{1}{2}$, and our method yields the sharp exponential order $e^{f/2}$.

Although the sharp exponential order is not needed in the applications in the following sections, it is analytically desirable and may be of independent interest.

\item If $f$ and $\operatorname{S}$ are nonnegative functions satisfying $f - |\nabla f|^2 = \operatorname{S}$ and $\frac{n}{2} - \Delta f = \operatorname{S}$, then the frequency method of \cite{CM21} gives sharp polynomial growth for any section $u$ with
\[
\langle -\Delta u + \n_{\n f} u, u\rangle \leq  \lambda |u|^2
\]
for some constant $\lambda$. See also the appendix of \cite{LZ25}. 
The above lemma also yields polynomial growth in integral form for $e^{f/2}\omega$ when $\tilde{\alpha}/a_1 = \frac{1}{4} $. 
\end{enumerate}
\end{remark}

We also need the following adaption which works for $f$-harmonic forms either in $L^2(M, dv_f)$ or on compact domain with Neumann boundary condition, and without curvature assumption. A key difference in the proof of the following lemma is that we avoid using the rough Laplacian.
\begin{lem}\label{lem: integral growth estimate - abs boundary condition}
Let $(M, g, dv_f= e^{-f}dv)$ be a smooth metric measure space, suppose there are nonnegative constants $a_1, \tilde{a} > 0$ and $b_1, \tilde{b} \geq 0$ such that 
\begin{enumerate}
\item $f \geq 0 $ and $f(x) \to \infty$ as $x\to \infty$;
\item $ |\n f|^2 \leq \frac{1}{2} ( a_1 f +b_1 )$;
\item $V_f |\alpha|^2+ \langle \mathcal{F}(\alpha), \alpha\rangle  \geq (\tilde{a} f - \tilde{b})|\alpha|^2$ for every $p$-form $\alpha$,
\end{enumerate}
where $V_f$ is defined in \eqref{eqn: definition of V_f} and $\mathcal{F}$ is given by \eqref{eqn: formula for F}. Let $\Omega = M$ or $\Omega = \{f \leq C\}$ for some regular value $C< \infty$ of $f$. Let $\omega$ be a $p$-form, such that $e^{f/2}\omega$ is $f$-harmonic on $\Omega$. If $\Omega = M$, we require $\omega \in L^2(M, dv)$; if $\Omega \subsetneq M$, we require that $e^{f/2}\omega$ satisfies the Neumann (absolute) boundary condition
\[
\quad \iota_{\nu} (e^{f/2}\omega)  = 0 \text{ along } \partial\Omega;
\]
where $\nu$ denotes the unit outer normal vector field on $\partial \Omega$. 
Then for any $\sigma > 0$, and 
\begin{equation}\label{eqn: choice of q}
q \geq \frac{ \sigma + \tilde{b} + a_1^{-1}\tilde{a} b_1 }{2\sqrt{a_1 \tilde{a}}} , 
\end{equation}
there are constants $C_1(q, \sigma, a_1, b_1, \tilde{a}, \tilde{b})$ and $C_2(q, a_1, b_1, \tilde{a}, \tilde{b}) $, such that 
\[
\sigma \int_{\Omega \cap \{f> C_1\}}  (1 + f)^{-2q} e^{2\sqrt{\tilde{a}/a_1}f}|\omega|^2 dv \leq C_2\int_{\Omega \cap \{f\leq C_1\}} |\omega|^2 dv.
\]
\end{lem}

Note that the factor $\frac{1}{2}$ in assumption $(ii)$ is non-essential, it is chosen merely for convenience so that parts of the proof of Lemma \ref{lem: integral growth estimate} can be applied verbatim.

\begin{proof}
Let $L$ be the operator defined in (\ref{eqn: definition of the operator L}), then we have $L \omega = 0$. Define $f_\theta$, $\phi$, $\omega_\theta$ and $L_\theta$ exactly as in the proof of Lemma \ref{lem: integral growth estimate}, recall that we have $L_\theta \omega_\theta = 0$. Using the weighted Weitzenb\"ock formula \eqref{eqn: weighted Weitzenbock formula}, we have $L = \Delta^d + V_f + \mathcal{F}$. Thus
\[
\begin{split}
0 =  & \int_\Omega \langle \phi(f_\theta) \omega_\theta, (\Delta^d  + V_f  + \mathcal{F} )(\phi(f_\theta)^{-1} \omega_\theta)\rangle dv \\
= &  \int_\Omega \langle \phi(f_\theta) \omega_\theta, \Delta^d (\phi(f_\theta)^{-1} \omega_\theta)\rangle dv + \int_\Omega (V_f |\omega_\theta|^2  + \langle \mathcal{F}(\omega_\theta), \omega_\theta \rangle ) dv.
\end{split}
\]
We treat the two cases $\Omega = M$ and $\Omega \subsetneq M$ separately.

\paragraph{\textbf{Case $\Omega = M$.}}
For any smooth function $\psi$ with compact support, 
\[
\begin{split}
& \int_M \psi \langle \phi(f_\theta) \omega_\theta, \Delta^d (\phi(f_\theta)^{-1} \omega_\theta)\rangle dv \\
= & \int \langle d(\psi \phi \omega_\theta) , d(\phi^{-1} \omega_\theta) \rangle +  \langle \delta (\psi \phi \omega_\theta ), \delta (\phi^{-1} \omega_\theta) \rangle \\
= &  \int \langle d\psi \wedge (\phi \omega_\theta )+ \psi d( \phi \omega_\theta) , d(\phi^{-1} \omega_\theta) \rangle +  \langle - \iota_{\n \psi} (\phi \omega_\theta) +  \psi \delta ( \phi \omega_\theta ), \delta (\phi^{-1} \omega_\theta) \rangle,\\
\end{split}
\]
where we used 
\[
\delta(\psi \phi \omega_\theta) =- \sum_i \iota_{e_i} \nabla_{e_i} (\psi \phi \omega_\theta) =-\sum_{i} \iota_{e_i} ( \n_i \psi \phi \omega_\theta + \psi \n_i(\phi \omega_\theta)) = -\iota_{\n \psi} (\phi \omega_\theta) + \psi \delta (\phi \omega_\theta)
\]
to get the last term.

For $R > 0$, choose $\psi = \psi(f)$ to be a cut off function compactly supported in the set $\{f < 2R\}$, with $\psi(f) = 1$ on $\{f \leq R\}$, such that $- \frac{2}{R} \leq \psi' \leq 0$. Note that $\phi^{-1}\omega_\theta = \omega$. By the assumption that $e^{f/2} \omega$ is $f$-harmonic, we have 
$0 = d(e^{f/2}\omega)= \frac{1}{2} e^{f/2} df \wedge \omega + e^{f/2} d\omega$ and $0 = \delta_f (e^{f/2} \omega) =  - \frac{1}{2}e^{f/2} \iota_{\n f} \omega + e^{f/2} \delta \omega + \iota_{\n f} (e^{f/2} \omega)$, hence
\[
d\omega = -\frac{1}{2} df \wedge \omega,
\]
and
\[
\delta \omega = -\frac{1}{2} \iota_{\n f} \omega.
\]
Using these relations, by calculating the terms involving $d \psi$ or $\iota_{\n \psi}$, we have
\[
\begin{split}
& \int_M \psi \langle \phi(f_\theta) \omega_\theta, \Delta^d (\phi(f_\theta)^{-1} \omega_\theta)\rangle dv \\
= &  \int_M -\frac{1}{2} \psi'  |df\wedge \omega_\theta|^2 + \psi \langle d( \phi \omega_\theta) , d(\phi^{-1} \omega_\theta) \rangle +\frac{1}{2} \psi' |\iota_{\n f} \omega_\theta|^2 +  \langle \psi \delta ( \phi \omega_\theta ), \delta (\phi^{-1} \omega_\theta) \rangle \\
\geq & \int_M \psi \langle d( \phi \omega_\theta) , d(\phi^{-1} \omega_\theta) \rangle +  \psi \langle \delta ( \phi \omega_\theta ), \delta (\phi^{-1} \omega_\theta) \rangle - \frac{2Ra_1 +b_1}{R}\int_{R<f<2 R} |\omega_\theta|^2,
\end{split}
\]
where the first term in the second line was dropped because $\psi' \leq 0$, and the third term in the second line was estimated by $|\iota_{\n f} \omega_\theta|^2 \leq |\n f|^2 |\omega_\theta|^2$. Note that $\int_{R< f< 2R} |\omega_\theta|^2 \to 0$ as $R \to \infty$ since $\omega$ is $L^2(dv)$-integrable and $\phi(f_\theta)$ is bounded. 

By direct calculation,
\[
\begin{split}
 \langle d (\phi(f_\theta) \omega_\theta), d (\phi(f_\theta)^{-1} \omega_\theta\rangle =& \langle d\phi \wedge \omega_\theta + \phi d\omega_\theta, -\phi^{-2} d\phi \wedge \omega_\theta + \phi^{-1} d\omega_\theta \rangle \\
 = & - |d \ln \phi|^2 |\omega_\theta|^2 + |d\omega_\theta|^2,
\end{split}
\]
and
\[
\begin{split}
\langle \delta (\phi(f_\theta) \omega_\theta), \delta (\phi(f_\theta)^{-1} \omega_\theta\rangle = & \langle -\iota_{\n \phi} \omega_\theta + \phi \delta \omega_\theta, \phi^{-2} \iota_{\n \phi} \omega_\theta + \phi^{-1} \delta \omega_\theta\rangle \\
= & -|\iota_{\n \ln \phi} \omega_\theta|^2 + |\delta \omega_\theta|^2 \\
\geq & - |\n \ln \phi|^2 |\omega_\theta|^2+ |\delta \omega_\theta|^2.
\end{split}
\]
Therefore, letting $R\to \infty$, we obtain
\[
\begin{split}
0 =  & \int_M \langle \phi(f_\theta) \omega_\theta, (\Delta^d  + V_f  + \mathcal{F} )(\phi(f_\theta)^{-1} \omega_\theta)\rangle dv \\
\geq & \int_M  \left( ( V_f - 2 |\n \ln \phi(f_\theta)|^2 ) |\omega_\theta|^2 + \langle \mathcal{F}(\omega_\theta), \omega_\theta \rangle + |d \omega_\theta|^2 +|\delta \omega_\theta|^2 \right)\\
\geq & \int_M  (\tilde{a} f- \tilde{b} - 2 |\n \ln \phi(f_\theta)|^2) |\omega_\theta|^2,
\end{split}
\]
where we used assumption (iii). Then the rest is the same as in the proof of Lemma \ref{lem: integral growth estimate}.

\paragraph{\textbf{Case $\Omega = \{f \leq C\}$ with Neumann (absolute) boundary condition.}}
By the Green's formula \eqref{eqn: Green's formula},
\[
\begin{split}
&\int_\Omega \langle \phi(f_\theta) \omega_\theta, \Delta^d (\phi(f_\theta)^{-1} \omega_\theta)\rangle dv \\
= & \int \langle \delta (\phi(f_\theta) \omega_\theta), \delta (\phi(f_\theta)^{-1} \omega_\theta\rangle + \langle d (\phi(f_\theta) \omega_\theta), d (\phi(f_\theta)^{-1} \omega_\theta\rangle \\
& + \int_{\partial \Omega} \bm{t} \delta(\phi(f_\theta)^{-1} \omega_\theta) \wedge * \bm{n} (\phi(f_\theta) \omega_\theta) - \bm{t} (\phi(f_\theta) \omega_\theta) \wedge * \bm{n} d (\phi(f_\theta)^{-1} \omega_\theta),
\end{split}
\]
where $\bm{t} \alpha = i^*_{\partial \Omega \to \Omega} \alpha$ is the tangential part of a differential form $\alpha$ on $\Omega$, and $\bm{n } \alpha = \alpha - \bm{t} \alpha $.

Let $\nu$ be the unit outer normal vector on $\partial \Omega$, the volume form on $\partial \Omega$ is given by $d\sigma = \iota_{\nu} dv$, one can check directly that for any $k$-form $\alpha$ and $(k+1)$-form $\beta$, we have 
\[
\bm{t} \alpha \wedge * \bm{n} \beta = \langle \alpha , \iota_{\nu}\beta\rangle d\sigma.
\]
Hence the boundary term becomes
\[
\int_{\partial \Omega} \langle \delta(\phi(f_\theta)^{-1} \omega_\theta), \phi(f_\theta) \iota_\nu \omega_\theta \rangle - \langle \phi(f_\theta) \omega_\theta, \iota_\nu d (\phi(f_\theta)^{-1} \omega_\theta) \rangle.
\]

The absolute boundary condition implies 
\[
\iota_\nu \omega_\theta = \phi(f_\theta) \iota_\nu \omega = \phi(f_\theta) e^{-f/2} \iota_\nu (e^{f/2} \omega) = 0. 
\]
Since $\Omega = \{f \leq C\}$, we can write $\nu = |\n f|^{-1} \n f$ on regular level sets of $f$, hence
\[
\iota_\nu d (\phi(f_\theta)^{-1} \omega_\theta) = \iota_\nu d \omega = - \frac{1}{2} \iota_\nu (df \wedge \omega) = -\frac{1}{2} (\iota_\nu df) \omega + \frac{1}{2} df \wedge \iota_\nu \omega= - \frac{1}{2} |\n f| \omega. 
\]
Thus the boundary term reduces to 
\[
\int_{\partial \Omega} \frac{1}{2} |\n f|  |\omega_\theta|^2.
\]
Dropping this nonnegative contribution, we obtain
\[
\int_\Omega \langle \phi(f_\theta) \omega_\theta, \Delta^d (\phi(f_\theta)^{-1} \omega_\theta)\rangle dv \\
\ge \int_\Omega \langle \delta (\phi(f_\theta) \omega_\theta), \delta (\phi(f_\theta)^{-1} \omega_\theta\rangle + \langle d (\phi(f_\theta) \omega_\theta), d (\phi(f_\theta)^{-1} \omega_\theta\rangle.
\]
Then the argument is the same as in the previous case. 
\end{proof}

\begin{remark}
If we consider the Dirichlet (relative) boundary condition 
\[
\bm{t} (e^{f/2} \omega)  = 0,
\]
instead of the Neumann (absolute) boundary condition, then we have $\bm{t} \omega_\theta = \phi \bm{t} \omega = 0$ and the boundary term becomes
\[
\int_{\partial \Omega} \bm{t} \delta(\phi(f_\theta)^{-1} \omega_\theta) \wedge * \bm{n} (\phi(f_\theta) \omega_\theta) = -\int_{\pd \Omega} \frac{1}{2} |\n f| |\iota_\nu \omega_\theta|^2 d\sigma ,
\]
which does not have the right sign. Hence the argument above breaks down for Dirichlet boundary condition.
\end{remark}

We conclude this section by proving the finite dimensionality of weighted $L^2$-harmonic forms.
\begin{thm}\label{thm-intro: finite dimension}
Let $(M, g, dv_f = e^{-f}dv)$ be a smooth metric measure space.
Suppose $f \geq 0$ and $f(x) \to \infty$ as $x\to \infty$; and suppose $|\nabla f|^2 \leq a f + b$ for some constants $a, b$.
\begin{enumerate}
\item If there exist constants $\tilde{a}> 0$ and $\tilde{b}$ such that $V_f |\omega|^2 + \langle \mathcal{R}_f(\omega), \omega\rangle \geq (\tilde{a} f - \tilde{b})|\omega|^2$ for every $p$-form $\omega$, then $\dim \mathcal{H}_f^p(M) < \infty$.
\item If there exist constants $\tilde{a}> 0$ and $\tilde{b}$ such that $V_f |\omega|^2 + \langle \mathcal{F}(\omega), \omega\rangle \geq (\tilde{a} f - \tilde{b})|\omega|^2$ for every $p$-form $\omega$, then $\dim \mathcal{H}_f^p(M) < \infty$; moreover, $\dim \mathcal{H}_{f, N}^p(\{f\leq R\})$ has a uniform upper bound for all $R$.
\end{enumerate}
\end{thm}
\begin{proof}
Using the estimates of either Lemma \ref{lem: integral growth estimate} or Lemma \ref{lem: integral growth estimate - abs boundary condition}, the finite dimensionality can be derived by the method in \cite{Li97} (see the proof of Theorem \ref{thm: estimate of Betti numbers for gradient Ricci shrinkers} for details of this method). For the uniform estimate of $\dim \mathcal{H}_{f,N}^p(\{f \leq R\})$, see the proof of Theorem \ref{thm: finite Betti numbers}.
\end{proof}

\section{Weighted $L^2$ cohomology and de Rham cohomology}\label{section: l2 cohomology and de Rham cohomology}

In this section we study the relation between the weighted $L^2$ cohomology 
and the de Rham cohomology. We first establish an injective map from the 
de Rham cohomology to the reduced weighted $L^2$-cohomology. 
Then we give conditions under which this map is an isomorphism. 
Furthermore, when the spectral condition of Lemma~\ref{lem: identification of reduced and unreduced L2 cohomology} 
is also satisfied, the unreduced and reduced weighted $L^2$-cohomologies coincide, 
yielding an isomorphism between the de Rham cohomology and the unreduced weighted 
$L^2$-cohomology as well. 
Theorem~\ref{thm-intro: hodge theorem for ricci shrinkers} is proved at the end of this section.

Assume there is a proper smooth function $\rho : M \to \mathbb{R}$, which is bounded from below and satisfies $\rho(x) \to + \infty$ as the point $x$ goes to infinity. Suppose there is a constant $\rho_0$ such that $|\n \rho| > 0$ on the set $\{\rho\geq \rho_0\}$. Let $N = \{x \in M: \rho(x) < \rho_0\}$. Then
$M \setminus N$ is diffeomorphic to $\partial N \times [\rho_0, \infty)$, where each slice $\partial N \times \{\rho\}$ is a level set of the function $\rho$. 

Let $i: N \hookrightarrow M$ denote the inclusion map. The pullback $i^*$ induces an isomorphism in de Rham cohomology:
\[
[i^*]: H_{dR}^*(M) \to H_{dR}^*(N). 
\]
(Here we use $i^*$ for the pullback of forms and $[i^*]$ for the induced map on cohomology.) Note that for maps $\phi$ and $\psi$, we have $[\psi^*]\circ[\phi^*] = [(\phi\circ\psi)^*]$.

We now construct a smooth retraction $r: M \to N$. Without loss of generality, we may assume that there exists a slightly smaller set $N_\sigma = \{x \in M : \rho(x) < \rho_0 - \sigma\}$ (with $\sigma>0$ small) such that $M \setminus N_\sigma$ is diffeomorphic to $\partial N \times [\rho_0-\sigma, \infty)$. The diffeomorphism is obtained via the flow generated by the vector field $\nabla\rho/|\nabla\rho|^2$. Specifically, let $\phi_s(x)$ be the family of diffeomorphisms defined by
\[
\frac{d}{ds} \phi_s(x) = \frac{\n \rho}{|\n \rho|^2} \circ \phi_s(x), \quad \phi_{\rho_0} = id, \quad s\in [\rho_0 -\delta, +\infty),
\]
for $x\in\partial N$. Then $\rho(\phi_s(x)) = s$, and the map
\[
\partial N \times [\rho_0 - \sigma, \infty) \to M \setminus N_\sigma, \quad  (x, \rho) \mapsto \phi_{\rho - \rho_0}(x),
\]
is a diffeomorphism. Its inverse gives the desired product structure.

Identify $M\setminus N_\delta$ with $\partial N \times [\rho_0 - \sigma, +\infty)$. Let $\psi(\rho)$ be a smooth increasing function such that 
\[
\psi(\rho)  = \begin{cases}
\rho; & \rho \in [\rho_0 - \sigma, \rho_0 -\sigma/2]; \\
\rho_0; &  \rho \in [\rho_0,+\infty).
\end{cases}
\]
 Then we can define the retration by 
\[
r(p) = \begin{cases}
p; & p \in N_\sigma; \\
(x, \psi(\rho)), &  p= (x, \rho) \in N \times [\rho_0 - \sigma, +\infty).
\end{cases}
\]

Then $r\circ i$ is smoothly homotopic to $id_N$, a homotopy $H: N \times [0,1] \to N$ can be given by 
\[
H(p, \theta) = \begin{cases}
p, & p \in N_\sigma; \\
(x, \theta \rho + (1-\theta) \psi(\rho)), & p = (x, \rho) \in N \times  [\rho_0 -\sigma, \rho_0];
\end{cases}
 \quad \theta \in [0, 1].
\]
This homotopy is well-defined and smooth. Consequently,
\[
 [i^*r^*] = id \quad on \quad H^*_{dR}(N). 
\]
Similarly, $i \circ r$ is smoothly homotopic to $id_M$, and 
\[
[r^* i^*] = id \quad on \quad H^*_{dR}(M). 
\]
This shows that 
\[
[i^*] : H^*_{dR}(M) \to H^*_{dR}(N)\quad and  \quad [r^*]: H^*_{dR}(N) \to H_{dR}^*(M)
\]
are mutual inverses, and both are isomorphisms.

Let $\lambda_1(\n\n \rho)\leq \cdots \leq \lambda_n(\n\n \rho) $ be the eigenvalues of the hessian $\n\n \rho$. Let $\lambda_1'(\n\n \rho) \leq \cdots \leq \lambda_{n-1}'(\n\n \rho)$ be the eigenvalues of $\n\n \rho$ restricted to the tangent space of the cross section $\partial N \times \{\rho\}$. 

\begin{lem}\label{lem: growth of the pullback form} Suppose $\Delta \rho - 2 \n_{\nu} \n_{\nu } \rho - 2 \sum_{k = 1}^p\lambda'_{k}(\n\n \rho) \leq  \left(1 - \frac{1+c_0}{\rho} \right)|\n \rho|^2 $ for some $c_0 > 0$ when $\rho > \rho_0$, where $\nu = \frac{\n \rho}{|\n \rho|}$, then for each smooth form $\omega \in \Lambda^p(M)$, we have
\[
\int_M |r^* i^* \omega|^2 e^{-\rho} dv < \infty. 
\]
\end{lem}
\begin{proof}
Identify $M \setminus N_\sigma$ with $\partial N \times [\rho_0-\sigma, +\infty)$, we can write 
\[
r(x, \rho) = (x, \psi(\rho)), \quad for \quad (x, \rho) \in \partial N \times [\rho_0-\sigma, +\infty).
\]
Let $x_1, ..., x_{n-1}$ be local coordinates on $\partial N$. On the cylinder $\partial N \times [\rho_0 - \sigma, +\infty)$, we have 
\[
r^* dx_i = d(x_i\circ r) = dx_i, \quad r^* d\rho = d \psi(\rho) = \psi'(\rho) d\rho ,
\]
in particular $r^* d\rho = 0$ for $\rho > \rho_0$. Note that $i^*\omega$ is just the restriction of $\omega$ on $N$.  Suppose 
$$\omega =\sum_{1\leq i_1< \cdots< i_p\leq n-1} \omega_{i_1,..., i_p}dx_{i_1} \wedge \cdots \wedge dx_{i_p} +  \sum_{1\leq i_1< \cdots< i_{p-1}\leq n-1} \omega'_{i_1, ..., i_{p-1}} d\rho \wedge dx_{i_1} \wedge \cdots \wedge dx_{i_{p-1}}, $$
 then for $\rho> \rho_0$, we have 
 \[
r^*i^* \omega (x, \rho)=\sum_{1\leq i_1< \cdots< i_p\leq n-1} \omega_{i_1, ..., i_p}(x, \rho_0) dx_{i_1} \wedge \cdots \wedge dx_{i_p},
 \]
 and
\[
|r^* i^* \omega|^2(x, \rho) = \sum_{1\leq i_1< \cdots< i_p\leq n-1}\sum_{1\leq j_1< \cdots< j_p\leq n-1} g^{i_1 j_1}(x, \rho)\cdots g^{i_p j_p}(x, \rho) \omega_{i_1, ..., i_p}(x, \rho_0) \omega_{j_1, ..., j_p}(x, \rho_0).
\]
Taking derivative of the metric tensor yields
\begin{equation}\label{eqn: growth of the metric along t}
\frac{d}{d\rho} g_{ij}(x, \rho) = (\mathcal{L}_{\frac{\n \rho}{|\n \rho|^2}} g)_{ij} = \frac{2\rho_{ij}}{|\n \rho|^2} - \frac{2\rho_i \rho_k \rho_{lj}g^{kl} + 2\rho_j \rho_k \rho_{li}g^{kl}}{|\n \rho|^4}, \quad  1\leq i, j \leq n,
\end{equation}
where we take $x_n = \rho$. By the choice of coordinates we have $g_{in} = 0$, $g^{in}= 0$ and $\rho_i = 0$ for $1\leq i \leq n-1$, and $g_{nn} = |\n \rho|^{-2}$, $\rho_n = 1$. Hence 
\begin{equation}\label{eqn: d/dt of metric}
\frac{d}{d\rho} g_{ij} = \frac{2\rho_{ij}}{|\n \rho|^2}\quad for \quad 1 \leq i, j \leq n-1, \quad and \quad \frac{d}{d\rho} g_{nn} = -\frac{2\rho_{nn}}{|\n \rho|^2}.
\end{equation}
Then by computing in normal coordinates, we can observe that 
\[
\begin{split}
\frac{d}{d\rho} |r^* i^* \omega|^2 = & - \sum \sum_{k = 1}^p g^{i_1 j_1}\cdots g^{i_k p}\left(\frac{2\rho_{pq}}{|\n \rho|^2} \right)g^{qj_k}\cdots g^{i_p j_p}(x, t) \omega_{i_1, ..., i_p}(x, \rho_0) \omega_{j_1, ..., j_p}(x, \rho_0) \\
= & -2\sum \left(\sum_{k = 1}^p \rho_{i_k i_k}\right)|\n \rho|^{-2} g^{i_1i_1} \cdots g^{i_pi_p}(x, \rho) |\omega_{i_1, ..., i_p}(x, \rho_0)|^2 \\
\leq & -2 |\n \rho|^{-2} \left( \sum_{k = 1}^p\lambda'_{k}(\n\n \rho)\right) |r^* i^* \omega|^2.
\end{split}
\]
Let $dv_N(x, \rho) = \sqrt{det(g_{ij})_{1\leq i,j \leq n-1}} dx^1 \wedge \cdots \wedge dx^{n-1}$ be the volume form on $N \times \{\rho\}$. By (\ref{eqn: d/dt of metric}), we can compute that 
\[
\frac{d}{d\rho} ( dv') = \frac{\sum_{i,j = 1}^{n-1} g^{ij} \rho_{ij}}{|\n \rho|^2 } (dv'), \quad and \quad \frac{d}{d\rho} |\n \rho|^{-1} = -\frac{g^{nn} \rho_{nn}}{|\n \rho|^2} |\n \rho|^{-1} .
\]
Hence 
\[
\begin{split}
\frac{d}{d\rho} ( |r^* i^* \omega|^2 |\n \rho|^{-1} dv' ) \leq  &  |\n \rho|^{-2} \left(-g^{nn} \rho_{nn} +  \sum_{i,j = 1}^{n-1} g^{ij} \rho_{ij} - 2 \sum_{k = 1}^p\lambda'_{k}(\n\n \rho)\right) ( |r^* i^* \omega|^2 |\n \rho|^{-1} dv' ) \\
\leq & |\n \rho|^{-2} \left(\Delta \rho - 2 \n_{\nu} \n_{\nu } \rho - 2 \sum_{k = 1}^p\lambda'_{k}(\n\n \rho) \right) ( |r^* i^* \omega|^2 |\n \rho|^{-1} dv' ).
\end{split}
\]

By the assumption $\Delta \rho - 2 \n_{\nu} \n_{\nu } \rho- 2 \sum_{k = 1}^p\lambda'_{k}(\n\n \rho) \leq  (1-\frac{1+c_0}{\rho}) |\n \rho|^2 $, we have 
\[
( |r^* i^* \omega|^2 |\n \rho|^{-1} dv' )(x, \rho) \leq e^{\rho -\rho_0}\left( \frac{\rho_0}{\rho} \right)^{1+c_0} ( |r^* i^* \omega|^2 |\n \rho|^{-1} dv' )(x, \rho_0),
\]
thus the integral
\[
\int_M |r^* i^* \omega|^2 e^{-\rho} dv = C + \int_{\rho_0}^\rho e^{-\rho} \int_{\partial N \times \{\rho\}}  |r^* i^* \omega|^2 |\n \rho|^{-1} dv' d\rho 
\]
is finite since $c_0 > 0$.
\end{proof}

\begin{remark}
Alternatively, we can assume that 
\[
- 2 \sum_{k = 1}^p\lambda'_{k}(\n\n \rho) \leq \pd_\rho v(x, \rho) |\n \rho|^2,
\]
for some function $v$. Then
\[
|r^* i^* \omega|^2(x, \rho) \leq e^{v(x, \rho) - v(x, \rho_0)} |r^* i^* \omega|(x, \rho_0).
\]
Note that $v(x, \rho_0)$ and $|r^*i^*\omega|^2(x, \rho_0)$ are bounded, thus the integral
\[
\begin{split}
\int_M |r^* i^* \omega|^2 e^{-\rho} dv = & C + \int_{\rho_0}^\rho e^{-\rho} \int_{\partial N \times \{\rho\}}  |r^* i^* \omega|^2 |\n \rho|^{-1} dv' d\rho \\
\leq & C + C \int_M e^{v(x, \rho)} e^{-\rho} dv.
\end{split}
\]
Therefore it is sufficient to assume that $\int_M e^{v} dv_\rho < \infty$.
\end{remark}

In the following theorem we construct an injective linear map from the de Rham cohomology to the weighted reduced $L^2$-cohomology.
\begin{thm}\label{lem: well-definedness and injectivity}
Let $(M, g, dv_f= e^{-f}dv)$ be a smooth metric measure space, suppose
\begin{enumerate}
\item $f \geq 0 $ and $f(x) \to \infty$ as $x\to \infty$;
\item $|\n f| > 0$ outside some compact domain;
\item $\Delta f - 2 \n_{\nu} \n_{\nu } f- 2 \sum_{k = 1}^q\lambda'_{k}(\n\n f) \leq   \left(1 - \frac{1+c_0}{f} \right)  |\n f|^2 $ for $q = p-1,p$ and $c_0 > 0$ outside some compact domain;
\end{enumerate}
then there is an injective linear map 
\[
[r^* i^*]: H_{dR}^p(M) \to \bar{H}_{(2)}^p(M, dv_f), \quad [\omega] \to [r^* i^* \omega]_{\overline{(2)}}.
\]
\end{thm}
\begin{proof}
Let $\rho = f$, the maps $i, r$ are defined as above.  

For any $[\omega] \in H_{dR}^p(M) $,  let's first verify that the map $ [\omega] \to [r^* i^* \omega]_{\overline{(2)}}$ is well-defined. Suppose $\omega' = \omega + d\alpha$. We have
\[
r^* i^* \omega' = r^* i^* \omega + r^* i^* d\alpha =  r^* i^* \omega + d (r^* i^* \alpha) .
\]
By Lemma \ref{lem: growth of the pullback form} and the assumption $(iii)$, the forms $r^*  i^* \omega$, $r^* i^* \alpha$ and $d (r^* i^* \alpha) = r^* i^* d \alpha$ are  $L^2(dv_f)$-integrable.
In particular, $d (r^* i^* \alpha)$ belongs to $\operatorname{im}(d_{p -1}|_{\Lambda_f^{p-1}(M)} )$. Hence the map $[\omega] \to [r^* i^* \omega]_{\overline{(2)}}$ is well-defined.

If for some $[\omega] \in H^p_{dR}(M)$, we have $[r^* i^* \omega]_{\overline{(2)}} = 0$, then Lemma \ref{lem: growth of the pullback form} and Lemma \ref{de Rham Lemma} shows that $[r^*i^*\omega]=0$, hence $[\omega]=0$ because $[r^*i^*] = id$ on the de Rham cohomology, this verifies injectivity. 
\end{proof}
Next, we give conditions under which this injective linear map is an isomorphism.
\begin{thm}\label{thm: identification of reduced L2 cohomology and de Rham cohomology}
Let $(M, g, dv_f= e^{-f}dv)$ be a smooth metric measure space, suppose there are constants $\epsilon_0, a_1, a_2> 0$ and $ b_1, b_2 \geq 0$, such that 
\begin{enumerate}
\item $f \geq 0 $ and $f(x) \to \infty$ as $x\to \infty$;
\item $ |\n f|^2 \leq a_1 f +b_1$ on $M$, and $|\n f|> \epsilon_0 $ outside some compact domain;
\item $\Delta f - 2 \n_{\nu} \n_{\nu } f- 2 \sum_{k = 1}^q\lambda'_{k}(\n\n f) \leq  \left(1 - \frac{1+c_0}{f} \right) |\n f|^2 $ for $q = p-1,p$ and $c_0 > 0$ outside some compact domain;
\item $V_f |\omega|^2 + \langle \mathcal{F}(\omega), \omega\rangle  \geq  (a_2 f - b_2)|\omega|^2$ for every $p$-form $\omega$.
\end{enumerate}
Then the map 
\[
[r^* i^*]: H_{dR}^p(M) \to \bar{H}_{(2)}^p(M, dv_f), \quad [\omega] \to [r^* i^* \omega]_{\overline{(2)}}
\]
is an isomorphism.
\end{thm}
\begin{proof}
By Theorem \ref{lem: well-definedness and injectivity}, the map is well-defined and injective. We only need to prove surjectivity. 

 For any $[\eta]_{\overline{(2)}} \in \bar{H}^p_{(2)}(M,dv_f)$, by (\ref{eqn: Hodge decomposition}), we can choose $\eta \in \mathcal{H}_f^p$. By homotopy equivalence we have $id - r^* i^* = K d + d K$, where the homotopy operator $K$ will be written explicitly below. Since $\eta$ is closed, we have $\eta - r^* i^* \eta = d K \eta$. By Lemma \ref{lem: growth of the pullback form}, the form $r^*  i^* \eta$ is $L^2(dv_f)$-integrable, hence $dK\eta$ is also $L^2(dv_f)$-integrable. We only need to show that $K\eta$ is $L^2(dv_f)$-integrable, then we have $[\eta]_{\overline{(2)}}  = [r^* i^* \eta]_{\overline{(2)}} $, which will imply surjectivity and finish the proof. 

Now we write down $K$ explicitly. Let $\rho = f$ and define $\psi(\rho)$ as in the discussion above. Let $H: M\times [0, 1] \to M$ be the smooth homotopy of $i\circ r$ with $id_M$ given by
\[
 H(p; \theta) = \begin{cases}
 p, & N_\sigma; \\
 (x, \theta \rho + (1-\theta) \psi(\rho)), & p = (x, \rho) \in \partial N \times [\rho_0 -\sigma, +\infty);
\end{cases}
\quad \theta \in [0, 1].
\]
 For any $\eta \in \Lambda^p(M)$, we can write 
\[
\eta = \eta_0 + \eta_1(\rho) + d\rho \wedge \eta_2(\rho), \quad \eta_1 \in \Lambda^{p}(\partial N), \quad \eta_2 \in \Lambda^{p-1} (\partial N),
\] 
where $\eta_1$ and $\eta_2$ are uniquely determined on $M \setminus N_\sigma$, then $\eta_0$ is simply taken to be $\eta - \eta_1 - d \rho \wedge \eta_2$, which vanishes on $M \setminus N_\sigma$.
For our purpose, it suffices to compute on $\partial N \times [\rho_0 -\sigma, +\infty)$ that
\[
H(\cdot, \theta)^* \eta = \eta_1\circ H(\cdot, \theta) + (\theta + (1-\theta)\psi'(\rho)) d\rho \wedge \eta_2\circ H(\cdot, \theta)  + (\rho - \psi(\rho)) d\theta \wedge \eta_2\circ H(\cdot, \theta) .
\]
By Cartan's magic formula, 
\[
\begin{split}
H(\cdot, 1)^* \eta - H(\cdot, 0)^* \eta = & \int_0^1 \left( \iota_{\frac{\pd}{\pd \theta}} d H(\cdot, \theta)^* \eta + d \iota_{\frac{\pd}{\pd \theta}}  H(\cdot, \theta)^* \eta \right) d\theta \\
= & \int_0^1  \iota_{\frac{\pd}{\pd \theta}}  H(\cdot, \theta)^* d\eta  d\theta + d \int_0^1   \iota_{\frac{\pd}{\pd \theta}}  H(\cdot, \theta)^* \eta  d\theta, \\
\end{split}
\]
hence by taking $K\eta = \int_{0}^1 \iota_{\frac{\pd}{\pd \theta}} H(\cdot, \theta)^* \eta d\theta$, we have $H(\cdot, 1)^* \eta - H(\cdot, 0)^* \eta = K d \eta + dK \eta  $. By the explicit formula for $H(\cdot, \theta)^* \eta$ above, the operator $K$ can be explicitly written as 
\begin{equation}\label{eqn: expression of K}
K\eta =\int_{0}^1 \iota_{\frac{\pd}{\pd \theta}} H(\cdot, \theta)^*  \eta d\theta = \int_0^1 (\rho - \psi(\rho)) \eta_2(\rho) d\theta =(\rho - \psi(\rho)) \eta_2(\rho) .
\end{equation}

By applying Lemma \ref{lem: integral growth estimate - abs boundary condition} to $\omega = e^{-f/2}\eta$, and absorbing the polynomial $(1+f)^l$ by the exponential weight, we see that for any $0\leq \epsilon < \sqrt{2 a_2/ a_1}$, and any $l \geq 0$,
\begin{equation}\label{eqn: weighted integrability of eta}
\int_M |\eta|^2 (1+f)^l e^{\epsilon f} dv_f < \infty .
\end{equation}

Since $\rho = f$, we have $\frac{\pd }{\pd \rho} = \frac{\n f}{|\n f|^2}$. For $\frac{\partial}{\partial x_i}$ tangent to the level sets of $\rho$ we have $g(\frac{\partial}{\partial x_i}, \frac{\partial}{\partial \rho}) = |\n f|^{-2} \frac{\partial f}{\partial x_i} = 0$, hence $\eta_1$ and $d\rho \wedge \eta_2$ are orthogonal. Consequently, both $\eta_1(\rho)$ and $d\rho \wedge \eta_2(\rho)$ satisfying the integrability in (\ref{eqn: weighted integrability of eta}). Then, take $l =2$ and $\epsilon = 0$, since $|d\rho \wedge \eta_2|^2 = |\n f|^2 |\eta_2|^2$, by assumption $(ii)$ we have
\[
\epsilon_0^2 \int_{M \setminus \Omega} |\eta_2|^2 (1+\rho)^2 dv_f\leq \int_{M \setminus \Omega} |d\rho \wedge \eta_2|^2 (1+\rho)^2 dv_f < \infty,
\]
for some compact set $\Omega$.
Together with the explicit expression \eqref{eqn: expression of K}, and note that $\psi(\rho)$ is bounded, this implies that $K\eta$ must be $L^2(dv_f)$-integrable. This finishes the proof.  
\end{proof}

Finally, under conditions that guarantee spectral discreteness for the weighted Hodge Laplacian, the reduced and unreduced weighted $L^2$-cohomologies coincide, which yields an isomorphism with the ordinary de Rham cohomology as well.
\begin{thm}\label{thm: identification of L2 cohomology and de Rham cohomology}
Let $(M, g, dv_f= e^{-f}dv)$ be a smooth metric measure space, suppose there are constants $\epsilon_0, a_1, a_2> 0$ and $ b_1, b_2 \geq 0$, such that 
\begin{enumerate}
\item $f \geq 0 $ and $f(x) \to \infty$ as $x\to \infty$;
\item $ |\n f|^2 \leq a_1 f +b_1$ on $M$, and $|\n f|> \epsilon_0$ outside some compact domain;
\item $\Delta f - 2 \n_{\nu} \n_{\nu } f- 2 \sum_{k = 1}^q\lambda'_{k}(\n\n f) \leq  \left(1 - \frac{1+c_0}{f} \right) |\n f|^2 $ for $q = p-1,p$ and $c_0 > 0$ outside some compact domain;
\item $V_f |\omega|^2 + \langle \mathcal{R}_f(\omega), \omega\rangle  \geq  (a_2 f - b_2)|\omega|^2$ for every $p$-form $\omega$.
\end{enumerate}
Then the map 
\[
[r^* i^*]: H_{dR}^p(M) \to H_{(2)}^p(M, dv_f), \quad [\omega] \to [r^* i^* \omega]_{(2)}
\]
is an isomorphism.
\end{thm}
\begin{proof}
We only need to modify the proof of Theorem \ref{thm: identification of reduced L2 cohomology and de Rham cohomology} slightly. 

By the assumptions, we can apply Lemma \ref{lem: identification of reduced and unreduced L2 cohomology} to get $H^p_{(2)}(M, dv_f)  = \bar{H}^p_{(2)}(M, dv_f)$, hence we can choose an $f$-harmonic representation for each cohomology class. The growth estimate need in the proof of Theorem \ref{thm: identification of reduced L2 cohomology and de Rham cohomology} is provided by Lemma \ref{lem: integral growth estimate} in this case.
\end{proof}

As an application of the results developed in this section, we now prove Theorem~\ref{thm-intro: hodge theorem for ricci shrinkers} stated in the introduction.
\begin{proof}[Proof of Theorem \ref{thm-intro: hodge theorem for ricci shrinkers}]
Let's verify the assumptions of Theorem \ref{thm: identification of reduced L2 cohomology and de Rham cohomology}.

By \cite{CZ2010}, there are constants $c_1, c_2$ depending only on $n$ such that
\[ \frac{1}{4}(r(x)-c_1)_+^2 \leq f(x) \leq \frac{1}{4}(r(x)+c_2)^2,
\]
where $r(x)$ is the distance to a fixed minimal point of $f$. We have $|\n f|^2 + \operatorname{S} = f$ by normalization, and the scalar curvature $\operatorname{S}$ is nonnegative \cite{Chen2009}, hence $f\geq 0$. 

Suppose $|\operatorname{Ric}|\leq \epsilon f$ outside some compact domain,  we have $\operatorname{S}\leq n\epsilon f$, and $|\n f|^2 \geq (1-n \epsilon f) > \frac{1}{2}f$ when $\epsilon < \frac{1}{2n}$. Hence Condition $(i)$ and $(ii)$ in Theorem \ref{thm: identification of reduced L2 cohomology and de Rham cohomology} are verified with $a_1 = 1$, $b_1=0$ and $\epsilon_0$ can be any positive number. 

By \eqref{eqn: Ricci shrinker equation}, condition $(iii)$ translates into 
\[
\frac{n}{2} - q + 2 (\operatorname{Ric}(\nu, \nu) + \sum_{k =1}^q \lambda'_k(\operatorname{Ric})) - \operatorname{S} \leq (1- \frac{1+c_0}{f}) (f - \operatorname{S}),
\]
for $q = p-1,p$, outside some compact domain. Add $\operatorname{S}$ to both sides, then the left hand side is bounded from above by $\frac{n}{2}+ 2n\epsilon f$, hence $(iii)$ is also verified when $\epsilon < \frac{1}{2n}$ with strict inequality.

By tracing \eqref{eqn: Ricci shrinker equation}, we have $\Delta f + \operatorname{S} = \frac{n}{2}$, hence $V_f = \frac{1}{4}(f+\operatorname{S} - n)$, and condition $(iv)$ in Theorem \ref{thm: identification of reduced L2 cohomology and de Rham cohomology} translates into 
\[
\left( \frac{1}{4} (f + \operatorname{S} -n) + \frac{q}{2}\right) |\omega|^2 - \langle \operatorname{Ric}(\omega), \omega\rangle \geq (a_2 f - b_2 ) |\omega|^2. 
\]
On compact domains we can always find $b_2$ depending on the local geometry such that the above hold.  Since $\langle \operatorname{Ric}(\omega), \omega\rangle \leq p\epsilon f|\omega|^2$ outside some compact domain, we observe that $(iv)$ holds when $\epsilon < \frac{1}{4n}$, and $a_2$ can be taken to be $\frac{1}{4} - n\epsilon$. Hence we can take $\epsilon = \frac{1}{5n}$, and the claim then follows from Theorem  \ref{thm: identification of reduced L2 cohomology and de Rham cohomology} and \eqref{eqn: reduced L2 chomology and harmonic forms}.
\end{proof}

\begin{remark}\label{rem: remark about AS00}
It is instructive to compare the results of this section with those of Ahmed--Stroock \cite{AS00}. 
There the weight function is taken to be $e^{-U}$, where $U$ satisfies $|\nabla U|^2 \ge c U^{1+\epsilon}$ outside a compact set (for some $c,\epsilon > 0$). 
In the present work, by contrast, we work with the weight $e^{-f}$ under the assumption $|\nabla f|^2 \le a_1 f + b_1$. 
These two regimes are therefore mutually exclusive. 
It is noteworthy that a Hodge-type isomorphism can nonetheless be established in both settings (under additional assumptions).
\end{remark}

To conclude this section, we remark that the identification of weighted $L^2$-cohomology with de Rham cohomology opens the door to generalizing vanishing theorems from closed manifolds to the present noncompact setting; see, for example, \cite{Lott03}. In recent years, there have been significant advances in understanding which curvature conditions imply vanishing for classical harmonic forms on closed manifolds \cite{HT25,PW21a,PW21b}, as well as for weighted harmonic forms \cite{PW20}. These conditions may also lead to interesting topological consequences in the noncompact case. We will not, however, pursue this direction here in order to keep the focus on our main results.

\section{Applications to Shrinkers}\label{section: applications}

In this section we apply the results developed in the previous sections to 
gradient Ricci shrinkers and to self-shrinkers of the mean curvature flow. 
We first establish general estimates for Betti numbers of smooth metric measure 
spaces, from which Theorem~\ref{thm-intro: Betti numbers of gradient Ricci shrinkers} 
follows. We then prove vanishing results for Betti numbers under curvature operator 
assumptions, establishing Theorem~\ref{thm-intro: vanishing of betti numbers}. 
The $(n-1)$-th Betti number and the number of ends are analyzed, and 
Theorem~\ref{thm-intro: number of ends} is proved. Finally, we demonstrate 
the flexibility of our method by proving a parallel result for self-shrinkers.

\subsection{Estimates of Betti numbers}
Before restricting to gradient Ricci shrinkers, we first prove an estimate for Betti numbers in the generality of smooth metric measure spaces. 
\begin{thm}\label{thm: finite Betti numbers}
Let $(M^n, g, dv_f= e^{-f}dv)$ be a smooth metric measure space, suppose there are constants $a_1, \tilde{a} > 0$ and $b_1, \tilde{b} \geq 0$, such that 
\begin{enumerate}
\item $f \geq 0 $ and $f(x) \to \infty$ as $x\to \infty$;
\item $ |\n f|^2 \leq \frac{1}{2} ( a_1 f +b_1 )$;
\item $V_f |\omega|^2+ \langle \mathcal{F}(\omega), \omega\rangle  \geq (\tilde{a} f - \tilde{b})|\omega|^2$ for every $p$-form $\omega$,
\end{enumerate}
where $V_f$ is defined in (\ref{eqn: definition of V_f}). 
Then there is a constant $C_p$ depending on $a_1, b_1, \tilde{a}, \tilde{b}, n, p$ and the geometry of the set $\{f < 2C_1\}$, where $C_1$ is the constant in Lemma \ref{lem: integral growth estimate - abs boundary condition}, such that 
$$b_p(M) \leq C_p.$$ 
Moreover, the restriction map $H^p_{dR}(M) \to H^p_{dR}(\{f\leq R\})$ is isomorphic when $R$ is sufficiently large. 
\end{thm}

\begin{proof}
The assumptions here is the same as in Lemma \ref{lem: integral growth estimate - abs boundary condition}. Let $C_1, C_2$ be the constants given by Lemma \ref{lem: integral growth estimate - abs boundary condition}. 

\textbf{Step 1:} Let $\Omega \subset M$ be any sublevel set $ \{f \leq C\}$ with $C > 2C_1$, where $C$ is a regular value of $f$ so that $\Omega$ has smooth boundary. We can estimate the dimension of $\mathcal{H}^p_{f, N}(\Omega)$, i.e. the space of $f$-harmonic $p$-forms on $\Omega$ with Neumann (absolute) boundary condition. Then by (\ref{eqn: hodge thm abs}) we get an estimate of $b_p(\Omega)$. 

For any $p$-form $\omega$ such that $e^{f/2}\omega \in \mathcal{H}_{f, N}^p(\Omega)$, we can apply Lemma \ref{lem: integral growth estimate - abs boundary condition}, where we take $\sigma = 1$ and take equality in (\ref{eqn: choice of q}), after absorbing the polynomial decay term by the exponential growth term in the weight, we get 
\begin{equation}\label{eqn: doubling estimate for omega}
\int_{\{f> C_1\} \cap \Omega} |\omega|^2 dv \leq C_2\int_{ \{f\leq C_1\} \cap \Omega} |\omega|^2 dv.
\end{equation}
By the equation \eqref{eqn: definition of the operator L},
\[
L \omega = -\Delta \omega + V_f \omega +  \mathcal{R}(\omega) + \mathcal{F}(\omega) = 0,
\]
we must have 
\[
\Delta |\omega|^2 \geq 2|\n \omega|^2 - C |\omega|^2 \quad \text{on} \quad \{f < 2C_1\}
\]
for some constant $C$ depending on the geometry of this bounded set.
Then we can apply the method of \cite{Li97} (see the proof of Theorem \ref{thm: estimate of Betti numbers for gradient Ricci shrinkers} for details of this method) to get an estimate of the dimension of $f$-harmonic forms
\[
\dim \mathcal{H}_{f, N}^p(\Omega) \leq C_p,
\]
where $C_p$ depends on $C_2$ and the geometry of the domain $\{f < 2 C_1\}$. 

\textbf{Step 2:} Let $R_1< R_2 < \cdots <R_k < \cdots$ be a sequence of regular values of $f$, with $R_1 > 2 C_1$, and $R_k \to \infty$ as $k \to \infty$. Let $\Omega_i = \{f \leq R_i\}$. By Step 1 and (\ref{eqn: hodge thm abs}) we have 
\[
\dim H^p_{dR}(\Omega_k) \leq C_p, \quad k = 1,2,...
\]
in particular, $C_p$ is independent of $k$.

Let $i_{k,l}: \Omega_k \to \Omega_l$ be the inclusion, where $k < l$, then there is an induced inverse system
\[
\cdots \xrightarrow{[i^*_{k+1,k+2}]} H^p_{dR}(\Omega_{k+1}) \xrightarrow{[i^*_{k,k+1}]} H^p_{dR}(\Omega_k) \xrightarrow{[i^*_{k-1,k}]} \cdots
\]

By \cite{Hatcher}(Proposition 3F.5, although not stated, the proof clearly works for real coefficients), we have 
\[
H^p_{dR}(M) \cong  \varprojlim H^p_{dR}(\Omega_k).
\]

Let $L = \varprojlim H^p_{dR}(\Omega_k)$ denote the inverse limit, it is the subgroup of $\prod_k H^p_{dR}(\Omega_k)$ consists of elements $(g_1, g_2, g_3, ...)$ satisfying $g_k = [i^*_{k,l}](g_l)$ for $k< l$. For each $k=1,2,...$, there is a natural projection $p_k : L \to H^p_{dR}(\Omega_k)$ given by $(g_1, g_2, g_3, ...) \mapsto g_k$. Denote $I_k$ to be the image of $p_k$, then $I_k$ is the subgroup of $H^p_{dR}(\Omega_k)$ consisting of cohomology classes that can be extended to $H^p_{dR}(\Omega_l)$ for all $l > k$, hence the maps
\[
[i^*_{k,l}] : I_l \to I_k, \quad k < l.
\]
are surjective. Consequently the sequence of dimensions $\dim I_k$, $k =1,2,... $ is nondecreasing, and uniformly bounded by $C_p$, hence must stablize when $k > N$ for some integer $N$ large enough. The restriction maps $[i^*_{N, l}]: I_l \to I_N$ become isomorphisms for $l > N$. Define a map $\phi: I_N \to L$ by the following way: for $g \in I_N$, let $\phi(g) = ([i^*_{1,N}](g),...,[i^*_{N-1, N}](g), g, [i^*_{N, N+1}]^{-1}(g), [i^*_{N, N+2}]^{-1}(g),...)$. Clearly $\phi$ is linear, injective and inverse to $p_N$. Therefore we have shown that $L \cong I_N$, hence $\dim L \leq C_p$.

Note that the isomorphism $H^p_{dR}(M) \to L$ is given by 
\[
[\omega] \mapsto ([i^*_1 \omega], [i^*_2 \omega], [i^*_3 \omega], ...),
\]
where $i_k : \Omega_k \to M$ is the inclusion, $k = 1,2,...$. Thus the proof above implies that $[i^*_k]$ is an isomorphism when $k> N$. By the arbitrariness of the sequence $\{R_i\}$, we see that the restriction map $H^p_{dR}(M) \to H^p_{dR}(\{f\leq R\})$ is isomorphic when $R$ is large enough.  
\end{proof}

We can make the estimates of Betti numbers more explicit for gradient Ricci shrinkers, this leads to Theorem \ref{thm-intro:  Betti numbers of gradient Ricci shrinkers} in the introduction, now let's restate it as:
\begin{thm}\label{thm: estimate of Betti numbers for gradient Ricci shrinkers}
For constants $\epsilon_0> 0$ and $c_0 \geq 0$, there is a constant $C_1$ depending on $n, \epsilon_0$ and $c_0$ such that the following holds.
Let $(M^n, g, f)$ be a complete gradient Ricci shrinker satisfying
\begin{enumerate}
\item
$\sum_{k=n-p+1}^{n} \lambda_{k}(\operatorname{Ric}) \leq  \frac{1}{4}\operatorname{S}+ ( \frac{1}{4} - \epsilon_0) f + c_0$ on $M$;
\item $\langle \mathcal{R}(\omega), \omega \rangle \geq - k |\omega|^2$ for every $p$-form $\omega$ on the domain $\{f < 2 C_1\}$.
\end{enumerate}
Then there is a constant $C(n,\epsilon_0, c_0, k)$, such that 
\[
b_p(M) \leq C(n, \epsilon_0, c_0, k)e^{-2\mu/n}  \binom{n}{p},
\]
where $\mu = \ln \int e^{-f} (4\pi)^{-\frac{n}{2}} dv$.
\end{thm}
\begin{proof}
By the proof of Theorem \ref{thm: finite Betti numbers}, we only need to estimate the dimension of $\mathcal{H}^p_{f, N}(\Omega)$ for $\Omega = \{f\leq R\}$, when $R$ is sufficiently large. We use the method of \cite{Li97} adapted to the current setting. 

The well-known result of Cao-Zhou \cite{CZ2010} tells that $f$ has quadratic growth. Normalize $f$ by possibly adding a constant so that $|\n f|^2+\operatorname{S} = f$ , since the scalar curvature is nonnegative on a gradient Ricci shrinker, we have $|\n f|^2 \leq f$. The shrinker equation $\operatorname{Ric} + \n\n f = \frac{1}{2} g$ directly implies that $\Delta f + \operatorname{S} = \frac{n}{2}$, hence we have 
\[
V_f = \frac{1}{4} |\n f |^2 - \frac{1}{2}\Delta f = \frac{1}{4}(f + \operatorname{S} - n).
\]
We only need to consider $1\leq p \leq n-1$ since $b_0(M) = 1$ and $b_n(M) = 0$. By the assumption $\sum_{k=n-p+1}^{n} \lambda_{k}(\operatorname{Ric}) \leq \frac{1}{4}\operatorname{S}+ ( \frac{1}{4} - \epsilon_0) f + c_0$, for any $p$-form $\omega$, we have
\[
\begin{split}
V_f |\omega|^2 + \langle \mathcal{F}(\omega), \omega \rangle \geq & \left( \frac{1}{4}(f+\operatorname{S} - n) + \frac{p}{2} - \sum_{k=n-p+1}^{n} \lambda_{k}(\operatorname{Ric}) \right) |\omega|^2 \\
\geq & \left( \epsilon_0 f - \frac{n}{4}  - c_0 \right) |\omega|^2.
\end{split}
\]
Thus we have verified the assumptions of Lemma \ref{lem: integral growth estimate - abs boundary condition} with $\tilde{a}=\epsilon_0$ and $\tilde{b}=\frac{n}{4}+c_0$. Let $\omega$ be a $p$-form such that $e^{f/2}\omega \in \mathcal{H}_{f, N}^p(\Omega)$, apply Lemma \ref{lem: integral growth estimate - abs boundary condition} to $\omega$, where we take $\sigma = 1$ and take equality in (\ref{eqn: choice of q}), after absorbing the polynomial decay term by the exponential growth term in the weight, we get 
\begin{equation}\label{eqn: doubling estimate for omega}
\int_{\{f> C_1\}\cap \Omega} |\omega|^2 dv \leq C_2\int_{ \{f\leq C_1\}\cap \Omega}  |\omega|^2 dv,
\end{equation}
where we take $R >> C_1$, and $C_1, C_2$ are constants from  Lemma \ref{lem: integral growth estimate - abs boundary condition}, note that they are determined by $\epsilon_0, c_0$ and $n$ in the current situation. We will need this later.

We can choose $\tilde{\omega}_1, ..., \tilde{\omega}_k$ to be a basis of $\mathcal{H}_{f, N}^p(\Omega)$, such that they are orthonormal with respect to the inner product defined as 
\[
(\alpha, \beta) = \int_{f\leq C_1} \langle \alpha, \beta \rangle dv_f.
\]
Note that this is indeed an inner product, for if an $f$-harmonic form vanishes on an open set, then it has to vanish on the whole domain. 
Let $\omega_i = e^{-f/2} \tilde{\omega}_i$, $i = 1,...,k$. Let $q$ be the maximum point of the function $\sum |\omega_i|^2$ on the compact set $\{f \leq C_1\}$. Note that $\sum |\omega_i|^2$ is invariant under orthonormal change of basis. By the method of \cite{Li80},  up to an orthonormal change of basis, the number of $\omega_i$'s that do not vanish at $q$ is at most $rank(\Lambda^p(M))$. Then
\[
\dim \mathcal{H}_{f, N}^p(\Omega)= \int_{\{f\leq C_1\}} \sum |\omega_i|^2 dv \leq rank(\Lambda^p(M)) Vol(\{f \leq C_1\}) \sup_{\omega}\sup_{\{f\leq C_1\}}|\omega|^2,
\]
where $\omega$ runs over all $e^{-f/2} \tilde{\omega}$, with $\tilde{\omega} \in \mathcal{H}_{f, N}^p(\Omega)$ and $(\tilde{\omega}, \tilde{\omega}) = 1$. 

By \cite{CZ2010} and \cite{HM2011}, geodecis balls on gradient Ricci shrinkers have polynomial volume growth $Vol(B(p, r)) \leq C(n) r^n$, where $p$ is a minimal point of $f$, thus we have the volume estimate
\[
Vol(\{f \leq C_1\}) \leq C(n) (2\sqrt{C_1} +c(n))^n.
\]

Then  we only need to estimate the supremum of $|\omega|$. We will use $(\ref{eqn: doubling estimate for omega})$ and a mean value inequality to finish the proof.

Since $\omega$ satisfies the equation 
\[
L \omega = -\Delta \omega + V_f \omega +  \mathcal{R}(\omega) + \mathcal{F}(\omega) = 0,
\]
by the assumption $\langle \mathcal{R}(\omega), \omega \rangle \geq - k |\omega|^2$ on $\{f < 2 C_1\}$, we have 
\[
\Delta |\omega|^2 \geq 2(\epsilon_0 f - \frac{n}{4} + \frac{p}{2} -c_0 - k) |\omega|^2 \geq - 2(  \frac{n}{4} +c_0 + k) |\omega|^2 ,
\]
on $\{f< 2 C_1\}$.
By Y.~Li and B.~Wang \cite{LW2020}, there is a Sobolev inequality
\[
\left( \int_M u^{\frac{2n}{n-2}} dv \right)^{\frac{n-2}{n}}\leq C(n)e^{-2\mu/n} \int_M |\n u|^2 + \operatorname{S} u^2 dv,
\]
for any compactly supported Lipschiz function $u$, where 
\[
\mu = \ln \int e^{-f} (4\pi)^{-\frac{n}{2}} dv
\]
is Perelman's entropy of the Ricci shrinker. Within $\{f<2 C_1\}$ we have $\operatorname{S}< 2C_1$, hence by the Nash-Moser iteration method (see, for example, \cite{Li12}) we get a mean value inequality for $|\omega|$,
\[
|\omega|^2(x) \leq C(n, C_1, k) e^{-2\mu/n}  \int_{B(x, 1)} |\omega|^2 dv
\]
for any $B(x, 1) \subset \{f< 2C_1\}$. Observe that $B(x, 1) \subset \{f< 2C_1\}$ for any $x \in\{ f \leq C_1\}$ when $C_1 > 1$. Hence by (\ref{eqn: doubling estimate for omega}) and the normalization $(\omega, \omega) = 1$ we get
\[
\sup_{\omega}\sup_{\{f< C_1\}}|\omega|^2 \leq C(n, C_1, k) (1+C_2)e^{-2\mu/n} .
\]
Combining the above estimates finish the proof. 
\end{proof}

Note that the condition $V_f |\omega|^2 + \langle \mathcal{R}_f(\omega), \omega\rangle  \geq  (a_2 f - b_2)|\omega|^2$ in Theorem \ref{thm: identification of L2 cohomology and de Rham cohomology} does not necessarily imply the condition 
$V_f |\omega|^2+ \langle \mathcal{F}(\omega), \omega\rangle  \geq (\tilde{a} f - \tilde{b})|\omega|^2$ in Theorem \ref{thm: finite Betti numbers}. Nevertheless, as an application of the isomorphism provided by Theorem \ref{thm: identification of L2 cohomology and de Rham cohomology}, we still have estimates of Betti numbers. 

\begin{cor} \label{cor: betti number estimates - R_f}
Under the assumption of Theorem \ref{thm: identification of L2 cohomology and de Rham cohomology}, there is a constant $C_p$ depending on $a_1, b_1, a_2, b_2, n, p$ and the geometry of the set $\{f < 2C_1\}$, where $C_1$ is the constant in Lemma \ref{lem: integral growth estimate}, such that 
$$b_p(M) \leq C_p.$$ 
\end{cor}
\begin{proof}
We only need to modify Step 1 in the proof of Theorem \ref{thm: finite Betti numbers} to obtain estimates of $\dim \mathcal{H}_f^p(M)$, then the estimates of Betti numbers follow from Theorem \ref{thm: identification of L2 cohomology and de Rham cohomology}. 

Use Lemma \ref{lem: integral growth estimate} instead of Lemma \ref{lem: integral growth estimate - abs boundary condition} to get the estimate 
\[
\int_{\{f> C_1\} } |\omega|^2 dv \leq C_2\int_{ \{f\leq C_1\} } |\omega|^2 dv.
\]
Note that the equation
\[
L \omega = -\Delta \omega + V_f \omega +  \mathcal{R}_f(\omega)= 0,
\]
and the assumption $V_f |\omega|^2 + \langle \mathcal{R}_f(\omega), \omega\rangle  \geq  (a_2 f - b_2)|\omega|^2$ imply that 
\[
\Delta |\omega|^2 \geq 2 |\n \omega|^2 - b_2 |\omega|^2.
\]
Then the proof is the same as in that of Theorem \ref{thm: finite Betti numbers}.
\end{proof}

On gradient Ricci shrinkers, we obtain the following estimate of Betti number. Recall that $\nu$ denotes the unit outer normal vector field on regular level sets of $f$. $\lambda'_1(\operatorname{Ric}) \leq \cdots \lambda'_{n-1}(\operatorname{Ric})$ denote the eigenvalues of $\operatorname{Ric}$ restricted to the tangent space of level sets of $f$.
\begin{cor}\label{cor: betti number estiamtes for Ricci shrinkers - R_f}
 Let $(M^n, g, f)$ be a complete gradient Ricci shrinker, suppose there are constants $\sigma, \epsilon_0, \epsilon_1 > 0$ and $c_0\geq 0$, such that 
\begin{enumerate}
\item $|\n f| > \sigma $ outside of some compact domain;
\item $\operatorname{Ric}(\nu, \nu) + \sum_{k = 0}^{q-1} \lambda'_{n-k}(\operatorname{Ric}) \leq (\frac{1}{2} - \epsilon_0) f$ outside some compact domain for $q = p-1,p$;
\item $(\frac{1}{4} - \epsilon_1)f |\omega|^2 + \frac{1}{4}\operatorname{S} |\omega|^2 + \langle \mathcal{R}_f(\omega), \omega \rangle \geq - c_0|\omega|^2$ for every $p$-form $\omega$ on M.
\end{enumerate}
Then there is a constant $C(n,\epsilon_0, \epsilon_1, c_0)$, such that 
\[
b_p(M) \leq C(n, \epsilon_0, \epsilon_1, c_0)e^{-2\mu/n}  \binom{n}{p},
\]
where $\mu = \ln \int e^{-f} (4\pi)^{-\frac{n}{2}} dv$.
\end{cor}
\begin{proof}
We only need to point out that on a gradient Ricci shrinker, using $\operatorname{Ric}_f = \frac{1}{2} g$, $|\n f|^2 + \operatorname{S}  = f$, we have $V_f = \frac{1}{4}|\n f|^2 - \frac{1}{2} \Delta f = \frac{1}{4}(f+\operatorname{S} -n)$, and the assumptions of Theorem \ref{thm: identification of L2 cohomology and de Rham cohomology} can be verified, hence $b_p(M) = \dim \mathcal{H}_f^p(M)$. 
The assumptions also imply 
\[
\Delta |\omega|^2 \geq 2 |\n \omega|^2 - (\frac{n}{4}+c_0) |\omega|^2,
\]
for every $p$-form $\omega$ such that $L\omega = 0$.
Then the argument is the same as in the proof of Theorem \ref{thm: estimate of Betti numbers for gradient Ricci shrinkers}.

Note that the constant $\sigma > 0$ in $(i)$ is only used to verify the assumptions of Theorem \ref{thm: identification of L2 cohomology and de Rham cohomology}, and is not used in the subsequent estimate. 
\end{proof}

For gradient Ricci shrinkers with bounded curvature, an estimate of the Betti number also follows from Theorem \ref{thm: identification of L2 cohomology and de Rham cohomology} and the main theorem in \cite{HO25}.

\subsection{Vanishing of Betti numbers and contractibility of gradient Ricci shrinkers}
The curvature operator on $2$-forms is defined by 
\begin{equation}\label{eqn: curvature operator}
  \operatorname{Rm}(\theta^i \wedge \theta^j) = \frac{1}{2}R_{ijkl} \theta^k \wedge \theta^l = \sum_{k <  l} R_{ijkl} \theta^k \wedge \theta^l.  
\end{equation}
We say $\operatorname{Rm} < K$ if the largest eigenvalue of $\operatorname{Rm}$ is less than $K$.

\begin{thm}\label{thm: vanishing of cohomology with curvature operator upper bounds}
Let $(M^n, g, f)$ be a gradient Ricci shrinker satisfying \eqref{eqn: Ricci shrinker equation}. For $1\leq p \leq n-1$, 
\begin{enumerate}
\item if $(p-1)\operatorname{Rm} < \frac{1}{2 }$, then $b_p(M) = 0$;
\item if $(p-1)\operatorname{Rm} \leq \frac{1}{2 }$, then $b_p(M) \leq \binom{n}{p}$, and elements of $\mathcal{H}^p_f(M)$ are parallel $p$-forms.
\end{enumerate}
\end{thm}
\begin{proof}
Let $\omega$ be a $p$-form. For fixed $(p-2)$ indices $i_1,..., i_{t-1}, i_{t+1}, ..., i_{s-1}, i_{s+1},..., i_p$, define a $2$-form $\bar{\omega}_{i_1...\widehat{i_t}...\widehat{i_s}...i_p} = \frac{1}{2}\omega_{i_1...j...k...i_p} \theta^j \wedge \theta^k$, where $j,k$ are in the $t$-th and $s$-th slots respectively. Then 
\[
\langle \operatorname{Rm}(\bar{\omega}_{i_1...\widehat{i_t}...\widehat{i_s}...i_p} ) , \bar{\omega}_{i_1...\widehat{i_t}...\widehat{i_s}...i_p}  \rangle  = \frac{1}{4} \sum_{j,k,a,b = 1}^n R_{jkab}\omega_{i_1...j...k...i_p}  \omega_{i_1...a...b...i_p} .
\]
Let $\lambda$ be the largest eigenvalue of $\operatorname{Rm}$.  By Lemma \ref{lem: decomposing the curvature term into Ric_f and curvature operator}, 
\[
\begin{split}
\langle \mathcal{R}_f (\omega), \omega \rangle = & \frac{1}{p!} \sum_{i_1,...,i_p} \frac{p}{2} |\omega_{i_1...i_p}|^2 -  \frac{1}{p!} \sum_{t\neq s} \sum_{i_1...\widehat{i_t}...\widehat{i_s}...i_p} 2 \langle \operatorname{Rm}(\bar{\omega}_{i_1...\widehat{i_t}...\widehat{i_s}...i_p} ) , \bar{\omega}_{i_1...\widehat{i_t}...\widehat{i_s}...i_p}  \rangle \\
\geq & \frac{1}{p!} \sum_{i_1,...,i_p} \left( \frac{p}{2}  - p(p-1)\lambda \right) |\omega_{i_1...i_p}|^2 . 
\end{split}
\]
Note that the scalar curvature is always nonnegative on a gradient Ricci shrinker, hence the assumption on $\operatorname{Rm}$ implies that the curvature is uniformly bounded, and the assumption of Theorem \ref{thm-intro: hodge theorem for ricci shrinkers} can be verified. Then the result follows from Theorem \ref{thm-intro: hodge theorem for ricci shrinkers} and standard arguments, namely, integration by parts yields 
\[
0 = \int_M \langle \omega, \Delta_f^d \omega \rangle dv_f = \int_M ( |\n \omega|^2 + \langle \mathcal{R}_f(\omega), \omega \rangle )  dv_f,
\]
for an $f$-harmonic form $\omega$,  hence both terms on the RHS has to vanish by the curvature assumption. 
\end{proof}

\begin{remark} If $\operatorname{Rm} < \frac{1}{2(n-1)}$, then $\operatorname{Ric} < \frac{1}{2}$ and $f$ is convex, consequently $M$ is diffeomorphic to $\mathbb{R}^n$, hence $b_1 = \cdots b_n = 0$. The theorem is nontrivial for $1 \leq p \leq n-1$.
\end{remark}

\begin{cor}\label{cor: contractibility of gradient Ricci shrinkers}
Let $(M^n, g, f)$ be a noncompact gradient shrinking Ricci soliton satisfying $\operatorname{Ric} + \n\n f = \frac{1}{2} g$, suppose the curvature operator $\operatorname{Rm} <\frac{1}{2(n-2)}$, and suppose the integral homology groups $H_k(M;\Z)$ are torsion-free for $k \geq 2$. Then 
$M$ is contractible. 
\end{cor}

\begin{proof}
Since gradient Ricci shrinkers in dimension $2$ are either the round sphere or the flat-Gaussian \cite{Hamilton95,Perelman03}, we only need to consider $n\geq 3$.
The assumption implies that $\operatorname{S}$ is bounded, then the equation $|\n f|^2 + \operatorname{S} = f$ implies that $\n f$ is nonsingular outside a compact domain, consequently $M$ is homotopy equivalent to a compact sublevel set of $f$, which has the homotopy type of a finite CW complex, hence the Euler characteristic is well-defined as an alternating sum of the number of cells, and is equal to the alternating sum of Betti numbers. By Theorem \ref{thm: vanishing of cohomology with curvature operator upper bounds}, we have $H^p(M) = 0$ for $p=1,...,n-1$. Since $M$ is noncompact, we also have $H^n(M) = 0$. Hence the only nonzero Betti number is $b_0(M)=1$ and the Euler characteristic $\chi(M) = 1$. 

We can show that $M$ is simply connected. By \cite{Wyl08}, $\pi_1(M)$ is finite. Let $\tilde{M}$ be the universal cover of $M$, then $\tilde{M}$ is a finite cover of $M$, and $\chi(\tilde{M}) = \chi(M) |\pi_1(M)|$. Since $\tilde{M}$ has the lifted Ricci shrinker structure with the same curvature condition, it also has $\chi(\tilde{M}) = 1$. Thus $|\pi_1(M)| = 1$. 

The Hurewicz theorem (\cite[Theorem 4.32]{Hatcher}) tells that if $M$ is $(k-1)$-connected for $k\geq 2$, then $\pi_k(M)\cong H_k(M;\Z)$. Since $b_1(M) = \cdots = b_n(M) = 0$, by the assumption that $H_k(M;\Z)$ are torsion-free, $k \geq 2$, we must have $H_k(M;\Z) = 0$ for each $k = 2,...,n$. Starting with $\pi_1(M) = \{1\}$, inductive application of the Hurewicz theorem implies that all the homotopy groups of $M$ vanish, and $M$ is contractible to a point.
\end{proof}

\subsection{The $(n-1)$-th Betti number and the number of ends.}

Let $*$ denote the Hodge star operator, we have 
\[
\alpha \wedge * \beta = \langle \alpha, \beta \rangle dv 
\]
for any $p$-forms $\alpha$ and $\beta$, where $dv$ is the Riemannian volume form. The Hodge star operator $*$ is parallel under the Levi-Civita covariant derivative, and it commutes with the Hodge Laplacian $\Delta^d$:
\[
\Delta^d * = * \Delta^d,
\]
hence it also commutes with the operator $\mathcal{R}$:
\[
\mathcal{R}(* \omega ) = * \mathcal{R} (\omega).
\]
However, the operator $\mathcal{F}$ does not commute with $*$. Let's denote an $(n-1)$-form as 
\[
\omega = \sum_{j =1}^n \omega_{1,2,...,\hat{j}, ..., n} \theta^{1}\wedge \theta^{2}\wedge \cdots \wedge \hat{\theta^j} \wedge\cdots \wedge \theta^{n},
\]
where $\hat{}$ means abscence of an item. Then by calculating at a point where $f_{ij}$ is diagonalized by the chosen frame, we have 
\[
\langle \mathcal{F}(\omega), \omega \rangle = \sum_{j = 1}^n\left( \sum_{s \neq j}f_{ss} \right) |\omega_{1,2,...,\hat{j}, ..., n}|^2.
\]
The Hodge star of $\omega$ can be written as
\begin{equation}\label{eqn: formula of *omega}
* \omega = \sum_{j =1}^n \omega_{1,2,...,\hat{j}, ..., n} (-1)^{n-j} \theta^{j}.
\end{equation}
Then we can write
\[
\mathcal{F}(* \omega) = \sum_{j,k =1}^n f_{jk}\omega_{1,2,...,\hat{j}, ..., n} (-1)^{n-j} \theta^{k}.
\]
Hence
\[
* \mathcal{F}(* \omega) = \sum_{j,k =1}^n f_{jk}\omega_{1,2,...,\hat{j}, ..., n} (-1)^{n-j + k - 1} \theta^{1}\wedge \theta^{2}\wedge \cdots \wedge \hat{\theta^k} \wedge\cdots \wedge \theta^{n},
\]
taking inner product with $\omega$, at a point where $f_{ij}$ is diagonalized, we get
\[
\langle * \mathcal{F}(* \omega), \omega \rangle = (-1)^{n-1} \sum_{j = 1}^n f_{jj} |\omega_{1,2,...,\hat{j}, ..., n}|^2.
\]

Then, since $\mathcal{R}_f(* \omega) = \operatorname{Ric}_f( * \omega)$, we have 
\[
\begin{split}
\langle \mathcal{R}_f(\omega), \omega \rangle = &  \langle (-1)^{n-1}* \mathcal{R}_f(* \omega) - (-1)^{n-1} * \mathcal{F}(* \omega) +  \mathcal{F}(\omega), \omega\rangle \\
= &  \langle \operatorname{Ric}_f(* \omega), * \omega \rangle + \sum_{j = 1}^n\left( -f_{jj} + \sum_{s \neq j}f_{ss} \right) |\omega_{1,2,...,\hat{j}, ..., n}|^2. \\
\end{split}
\]
On gradient Ricci shrinkers we have $\operatorname{Ric}_f = \frac{1}{2} g$, hence the above equation can be written as
\begin{equation}\label{eqn: formula for (R_f w,w) on n-1 form w}
\langle \mathcal{R}_f(\omega), \omega \rangle = \sum_{j = 1}^n\left( \frac{n-1}{2} +2R_{jj} -\operatorname{S} \right) |\omega_{1,2,...,\hat{j}, ..., n}|^2.
\end{equation}

\begin{thm}\label{thm: number of ends}
Let $(M^n, g, f)$ be a complete noncompact gradient Ricci shrinker satisfying \eqref{eqn: Ricci shrinker equation}. If it satisfies 
\begin{equation}\label{eqn: Einstein tensor lower bound}
    \operatorname{Ric} - \frac{1}{2}\operatorname{S}g \geq - \frac{n-1}{4}g,
\end{equation}
then one of the following cases must hold:
\begin{enumerate}
\item $b_{n-1}(M) = 0$ and $(M, g)$ has only one end;
\item $b_{n-1}(M) = 1$ and $(M, g)$ splits as $(N, g_N) \times (\mathbb{R}, g_{Eucl})$, where $(N, g_N)$ is a compact Einstein manifold.
\end{enumerate}
\end{thm}
\begin{proof}
Since gradient Ricci shrinkers in dimension $2$ are either the round sphere or the flat-Gaussian \cite{Hamilton95,Perelman03}, we can assume $n\geq 3$. By \cite{Chen2009}, $\operatorname{S}\geq 0$ on a gradient Ricci shrinker, \eqref{eqn: Einstein tensor lower bound} implies that $\operatorname{Ric}$ is bounded from below; by taking trace of \eqref{eqn: Einstein tensor lower bound}, we see that $\operatorname{S}$ is also bounded from above when $n\geq 3$, consequently the Ricci curvature must be bounded from both above and below. Thus the assumptions of Theorem \ref{thm: identification of reduced L2 cohomology and de Rham cohomology} or Theorem \ref{thm-intro: hodge theorem for ricci shrinkers} can be verified, consequently we have $H^{n-1}_{dR}(M) \cong \mathcal{H}^{n-1}_f(M)$.

By the curvature assumption and (\ref{eqn: formula for (R_f w,w) on n-1 form w}), standard arguments yield that all $L^2(dv_f)$-integrable $f$-harmonic $(n-1)$-forms are parallel and make the curvature term in (\ref{eqn: formula for (R_f w,w) on n-1 form w}) vanish. 

If $\dim \mathcal{H}^{n-1}_f(M) > 0$, there exists a nonzero parallel $(n-1)$-form $w \in \mathcal{H}^{n-1}_f(M)$, $*\omega$ is a parallel $1$-form, so the metric locally splits as a product, hence $\operatorname{Ric}(*\omega, *\omega) = 0$. Moreover, by (\ref{eqn: formula for (R_f w,w) on n-1 form w}), we must have $\operatorname{Ric}(*\omega, *\omega)/|*\omega|^2 - \frac{1}{2}\operatorname{S} = -\frac{n-1}{4}$. Thus the scalar curvature is constant $\operatorname{S} = \frac{n-1}{2}$.

The universal cover $(\tilde{M}, \tilde{g})$ with the pullback metric also has a parallel vector field, by the de Rham decomposition theorem, it must split as a Riemannian product $\tilde{N} \times \R^k$, where $\tilde{N}$ is a product of irreducible factors, and $k \geq 1$. We claim that $k = 1$. To see this, note that the universal cover also has a gradient Ricci shrinker structure with the pullback potential function $\tilde{f}$, which also has constant scalar curvature $\operatorname{S}=\frac{n-1}{2}$, by \cite[Lemma 3.1 and equation (3.3)]{CZ2010}, we can derive that 
\[
nV(r) - rV'(r) = (n-1) V(r) - (n-1) \frac{2}{r} V'(r),
\]
where $V(r)$ is the volume of the set $\{2\sqrt{f}< r\}$, and is approximately the volume of a geodesic ball with radius $r$ when $r$ is large. Solving this ODE yields $V(r) = C\sqrt{r^2-2(n-1)}$ for some constant $C$, hence $V(r)$ has linear growth when $r \to \infty$. On the other hand the Riemannian product $\tilde{N} \times \R^k$ has volume growth at least in the order of $r^k$, thus we must have $k\leq 1$. Consequently $\dim \mathcal{H}_f^{n-1}(M) = 1$.

By the de Rham decomposition theorem, the action of $\pi_1(M)$ on $\tilde{M}$ preserves the Euclidean factor $\R$, thus there is an induced isometric action by $\pi_1(M)$ on $\R$. Note that $\pi_1(M)$ is finite \cite{Wyl08}, and the quotient space is smooth, thus the induced action must be trivial on the $\R$ factor. Therefore $(M, g)$ itself splits as a product $N \times \R$ isometrically. This product structure implies that $b_{n-1}(M) = b_{n-1}(N)$, hence $N$ must be compact (since a noncompact $(n-1)$-dimensional manifold must have vanishing $b_{n-1}$). 

Now $({N}, {g}_{{N}},{f}_N)$ is a compact gradient Ricci shrinker with constant scalar curvature. By maximum principle applied to the equation $\Delta f = \frac{n}{2} - \operatorname{S}$, we see that a compact gradient Ricci shrinker with constant scalar curvature must have constant potential function $f$, hence is Einstein. Therefore $({N}, {g}_{{N}})$ must be an Einstein manifold with $\operatorname{Ric}({g}_{{N}}) = \frac{1}{2} {g}_{{N}}$.

If $\dim \mathcal{H}^{n-1}_f(M) = 0$, then $b_{n-1}(M) = 0$, by Poincare duality we have $H_0^1(M) \cong H^{n-1}(M)$, where $H_0^1(M)$ is the compactly supported first cohomology group, it is well-known that $\sharp ends(M) \leq \dim H_0^1(M)+1$ (see for example \cite{Carron2002}), hence $M$ has only one end. 
\end{proof}

Examples satisfying the assumption of Theorem~\ref{thm: number of ends} 
include the product shrinkers $N^{n-k} \times \mathbb{R}^k$ with $N^{n-k}$ compact Einstein and $k \ge 1$. The case $k = 1$ attains the curvature equality and realizes the splitting in part $(ii)$. 

From the above proof, and Theorem \ref{thm: estimate of Betti numbers for gradient Ricci shrinkers}, we see that the lower bound of the Einstein tensor $\operatorname{Ric} - \frac{1}{2}\operatorname{S}$ controls the number of ends for gradient Ricci shrinkers.

\subsection{An application to Self-shrinkers of mean curvature flow}
The tools developed in the present article also apply to self-shrinkers of the mean curvature flow. As a demonstration, we prove a result similar to Theorem \ref{thm: number of ends}.

Let $\Sigma \subset \mathbb{R}^{n+1}$ be an $n$-dimensional self-shrinker of the mean curvature flow, it satisfies 
\[
\vH= -\frac{1}{2} \vx^\perp,
\]
where $\vH$ is the mean curvature vector and $\vx$ is the position vector. 

Denote $H = |\vH|$. Let $A$ be the second fundamental form as a symmetric tensor, whose components are denoted as $a_{ij}$, $1\leq i,j \leq n$. 
And let $g$ be the induced Riemannian metric. 

Define $f = \frac{1}{4}|\vx|^2$, we can derive by calculation that
\[
|\n^\Sigma f|^2 = f - H^2,
\] 
\[
\n^\Sigma \n^\Sigma f = \frac{1}{2} g - HA.
\]
By the Gauss equation $R_{ijkl} = a_{ik}a_{jl} - a_{il} a_{jk}$, the Ricci curvature of $\Sigma$ can be written as 
\[
\operatorname{Ric}= HA - A^2,
\]
where $A^2$ has components $(A^2)_{ij} = a_{ik} a^k_j$. Hence 
\[
\operatorname{Ric}_f = \frac{1}{2} g - A^2.
\]
Similar to Theorem \ref{thm: number of ends}, we can prove:
\begin{thm}\label{thm: number of ends for MCSSh}
Let $\Sigma \subset \mathbb{R}^{n+1}$ be a complete noncompact properly embedded $n$-dimensional self-shrinker of the mean curvature flow. If it satisfies 
\[
(A-Hg)^2 \leq \frac{n-1}{2} g,
\]
where $(A-Hg)^2$ denotes the symmetric tensor with components $(A-Hg)^2_{ij} = (A-Hg)_{ik} (A-Hg)^k_j$, then one of the following cases must hold:
\begin{enumerate}
\item $b_{n-1}(M) = 0$ and $\Sigma$ has only one end;
\item $b_{n-1}(M) = 1$ and $\Sigma$ is the cyliner $\mathbb{S}^{n-1}(\sqrt{2(n-1)}) \times \R$.
\end{enumerate}
\end{thm}
\begin{proof}
By properness of the embedding, $f = \frac{1}{4} |\vx|^2$ is proper and goes to infinity as $\vx \to \infty$ The assumption implies that 
\[
- \sqrt{\frac{n-1}{2}} g \leq A - Hg \leq \sqrt{\frac{n-1}{2}} g,
\]
taking trace of this inequality implies that $H$ and $A$ are bounded. Then the assumptions of Theorem \ref{thm: identification of reduced L2 cohomology and de Rham cohomology} can be verified, hence $H^{n-1}_{dR}(\Sigma) = \mathcal{H}^{n-1}_f(\Sigma)$.

By (\ref{eqn: formula for (R_f w,w) on n-1 form w}) and the equations of the self-shrinker, the curvature term for an $(n-1)$-form can be written in orthonormal frame as 
\[
\begin{split}
\langle \mathcal{R}_f(\omega), \omega \rangle = & \sum_{j = 1}^n\left( \frac{n-1}{2} - \sum_{k=1}^n a_{jk} a_{kj} + 2Ha_{jj} -H^2\right) |\omega_{1,2,...,\hat{j}, ..., n}|^2 \\
 =& \sum_{j = 1}^n\left( \frac{n-1}{2} - (A-Hg)^2_{jj}\right) |\omega_{1,2,...,\hat{j}, ..., n}|^2.
\end{split}
\]
The assumption guarantees that this curvature term is nonnegative, and all $L^2(dv_f)$-integrable $f$-harmonic $(n-1)$-forms are parallel. 

If $\dim \mathcal{H}_f^{n-1}(\Sigma) = 0$ then $\Sigma$ has only one end by the same argument as in Theorem \ref{thm: number of ends}.

If $\dim \mathcal{H}_f^{n-1}(\Sigma) \geq 1$. Let $\omega \in \dim \mathcal{H}_f^{n-1}(\Sigma)$ be a nontrivial $f$-harmonic form, then $*\omega$ is a parallel $1$-form and let $\vv$ be its dual vector field. The metric $g$ locally splits as a product metric and $\operatorname{Ric}(\vv,\vv) = 0$. Note that $\operatorname{Ric} = (Hg -A)A$, and vanishing of the curvature term implies that $(A-Hg)^2(\vv, \vv) = \frac{n-1}{2}$, these two equation together imply that 
\[
H^2 = \frac{n-1}{2}+ HA(\vv, \vv) = \frac{n-1}{2}+ A^2(\vv, \vv) \geq \frac{n-1}{2}, 
\]
hence by choosing appropriate orientation we have $H \geq \sqrt{\frac{n-1}{2} } > 0$. Then by \cite{CZ13}, $\Sigma$ has linear volume growth. By \cite{CM12} (see also \cite{Ri14}), a properly embedded self-shrinker with nonnegative mean curvature and polynomial volume growth must be a cylinder $\mathbb{S}^k(\sqrt{2k}) \times \R^{n-k}$, in our case, clearly $k = n-1$ is the only possibility. 
\end{proof}






\begin{thebibliography}{99}
\footnotesize

\bibitem{Agmon82} Agmon, S. \emph{Lectures on exponential decay of solutions of second-order elliptic equations: Bounds on eigenfunctions of N-body Schr\"odinger operators}. Math. Notes 29, Princeton University Press, Princeton, 1982.

\bibitem{AS00} Ahmed, Z. M.; Stroock, D. W. \emph{A Hodge theory for some non-compact manifolds}. J. Differential Geom. \textbf{54} (2000), no. 1, 177–225.

\bibitem{APS75} Atiyah, M. F.; Patodi, V. K.; Singer, I. M. \emph{Spectral asymmetry and Riemannian geometry I}. Math. Proc. Cambridge Philos. Soc. \textbf{77} (1975), 43–69.

\bibitem{Anderson87} Anderson, M. \emph{$L^2$ harmonic forms on complete Riemannian manifolds}. Geometry and Analysis on Manifolds (Katata/Kyoto, 1987), Lecture Notes in Math., 1339, Springer, Berlin, 1988, pp. 1–19.

\bibitem{AK22} Angenent, S.; Knopf, D. \emph{Ricci solitons, conical singularities, and nonuniqueness}. Geom. Funct. Anal. \textbf{32} (2022), no. 3, 411–489.

\bibitem{Bamler2020} Bamler, R. H. \emph{Structure theory of non-collapsed limits of Ricci flows}. arXiv:2009.03243 (2020).

\bibitem{Bamler2023} Bamler, R. H. \emph{Compactness theory of the space of Super Ricci flows}. Invent. Math. \textbf{233} (2023), 1121–1277.

\bibitem{BCCD} Bamler, R. H.; Cifarelli, C.; Conlon, R. J.; Deruelle, A. \emph{A new complete two-dimensional shrinking gradient K\"ahler-Ricci soliton}. Geom. Funct. Anal. \textbf{34} (2024), no. 2, 377–392.

\bibitem{BB2026} Bertellotti, A.; Buzano, R. \emph{Ends of (singular) Ricci shrinkers}. Selecta Math. (N.S.) \textbf{32} (2026), no. 1, Paper No. 1, 46 pp.

\bibitem{Bue99} Bueler, E. \emph{The heat kernel weighted Hodge Laplacian on non compact manifolds}. Trans. Amer. Math. Soc. \textbf{351} (1999), no. 2, 683–713.

\bibitem{Bullock02} Bullock, S. S. \emph{Gaussian weighted unreduced $L^2$ cohomology of locally symmetric spaces}. New York J. Math. \textbf{8} (2002), 241–256.

\bibitem{CaoZhang11} Cao, X.; Zhang, Q. \emph{The conjugate heat equation and ancient solutions of the Ricci flow}. Adv. Math. \textbf{228} (2011), no. 5, 2891–2919.

\bibitem{CZ2010} Cao, H.-D.; Zhou, D. \emph{On complete gradient shrinking Ricci solitons}. J. Differential Geom. \textbf{85} (2010), no. 2, 175–185.

\bibitem{Carron2002} Carron, G. \emph{$L^2$ harmonic forms on non-compact Riemannian manifolds}. Surveys in Analysis and Operator Theory (Canberra, 2001), Proc. Centre Math. Appl. Austral. Nat. Univ., 40, Austral. Nat. Univ., Canberra, 2002, pp. 49–59.

\bibitem{CFSZ2020} Chow, B.; Freedman, M.; Shin, H.; Zhang, Y. \emph{Curvature growth of some 4-dimensional gradient Ricci soliton singularity models}. Adv. Math. \textbf{372} (2020), 107303, 17 pp.

\bibitem{CCD} Cifarelli, C.; Conlon, R. J.; Deruelle, A. \emph{On finite time Type I singularities of the K\"ahler-Ricci flow on compact K\"ahler surfaces}. J. Eur. Math. Soc. (JEMS) (2024), in press. arXiv:2203.04380.

\bibitem{CDS} Conlon, R. J.; Deruelle, A.; Sun, S. \emph{Classification results for expanding and shrinking gradient K\"ahler-Ricci solitons}. Geom. Topol. \textbf{28} (2024), no. 1, 267–351.

\bibitem{CZ2022} Chan, P.-Y.; Zhu, B. \emph{On a dichotomy of the curvature decay of steady Ricci solitons}. Adv. Math. \textbf{404} (2022), part B, Paper No. 108458, 40 pp.

\bibitem{Chen2009} Chen, B.-L. \emph{Strong uniqueness of the Ricci flow}. J. Differential Geom. \textbf{82} (2009), no. 2, 363–382.

\bibitem{CZ13} Cheng, X.; Zhou, D. \emph{Volume estimate about shrinkers}. Proc. Amer. Math. Soc. \textbf{141} (2013), no. 2, 687–696.

\bibitem{CZ17} Cheng, X.; Zhou, D. \emph{Eigenvalues of the drifted Laplacian on complete metric measure spaces}. Commun. Contemp. Math. \textbf{19} (2017), no. 1, 1650001.

\bibitem{CM12} Colding, T. H.; Minicozzi, W. P., II. \emph{Generic mean curvature flow I; generic singularities}. Ann. of Math. (2) \textbf{175} (2012), no. 2, 755–833.

\bibitem{CM21} Colding, T. H.; Minicozzi, W. P., II. \emph{Optimal growth bounds for eigenfunctions}. arXiv:2109.04998 (2021).

\bibitem{Dai2011} Dai, X. \emph{An introduction to $L^2$ cohomology}. Topology of Stratified Spaces, Math. Sci. Res. Inst. Publ., 58, Cambridge Univ. Press, Cambridge, 2011, pp. 1–12.

\bibitem{DY23} Dai, X.; Yan, J. \emph{Witten deformation for noncompact manifolds with bounded geometry}. J. Inst. Math. Jussieu \textbf{22} (2023), no. 2, 643–680.

\bibitem{deRham1973} de Rham, G. \emph{Vari\'et\'es diff\'erentiables. Formes, courants, formes harmoniques}. 3rd ed., Hermann, Paris, 1973.

\bibitem{EMT11} Enders, J.; M\"uller, R.; Topping, P. \emph{On type-I singularities in Ricci flow}. Comm. Anal. Geom. \textbf{19} (2011), no. 5, 905–922.

\bibitem{Hatcher} Hatcher, A. \emph{Algebraic topology}. Cambridge Univ. Press, Cambridge, 2002.

\bibitem{Hamilton95} Hamilton, R. \emph{The formation of singularities in the Ricci flow}. Surveys in Differential Geometry, Vol. II, Int. Press, Cambridge, MA, 1995, pp. 7–136.

\bibitem{HM2011} Haslhofer, R.; M\"uller, R. \emph{A compactness theorem for complete Ricci shrinkers}. Geom. Funct. Anal. \textbf{21} (2011), no. 5, 1091–1116.

\bibitem{HO25} He, F.; Ou, J. \emph{Dimension estimate and existence of holomorphic sections with polynomial growth on gradient K\"ahler Ricci shrinkers}. Int. Math. Res. Not. IMRN \textbf{2025} (2025), no. 23, 1–29.

\bibitem{HN14} Hein, H.-J.; Naber, A. \emph{New logarithmic Sobolev inequalities and an $\epsilon$-regularity theorem for the Ricci flow}. Comm. Pure Appl. Math. \textbf{67} (2014), no. 9, 1543–1561.

\bibitem{HW2022} Hua, B.; Wu, J. \emph{Gap theorems for ends of smooth metric measure spaces}. Proc. Amer. Math. Soc. \textbf{150} (2022), no. 11, 4947–4957.

\bibitem{HT25} Huang, T.; Tan, Q. \emph{Curvature operator and Euler number}. Calc. Var. Partial Differential Equations \textbf{64} (2025), Paper No. 205.

\bibitem{Li80} Li, P. \emph{On the Sobolev constant and the $p$-spectrum of a compact Riemannian manifold}. Ann. Sci. \'Ecole Norm. Sup. (4) \textbf{13} (1980), no. 4, 451–468.

\bibitem{Li97} Li, P. \emph{Harmonic sections of polynomial growth}. Math. Res. Lett. \textbf{4} (1997), no. 1, 35–44.

\bibitem{Li12} Li, P. \emph{Geometric analysis}. Cambridge Stud. Adv. Math., 134, Cambridge Univ. Press, Cambridge, 2012.

\bibitem{LW2020} Li, Y.; Wang, B. \emph{Heat kernel on Ricci shrinkers}. Calc. Var. Partial Differential Equations \textbf{59} (2020), no. 6, Paper No. 194, 84 pp.

\bibitem{LW2024} Li, Y.; Wang, B. \emph{On K\"ahler–Ricci shrinker surfaces}. Acta Math., in press (2024).

\bibitem{LZ25} Li, Y.; Zhang, W. \emph{On the rigidity of Ricci shrinkers}. arXiv:2305.06143 (2023).

\bibitem{Lott03} Lott, J. \emph{Some geometric properties of the Bakry-Emery-Ricci tensor}. Comment. Math. Helv. \textbf{78} (2003), 865–883.

\bibitem{MS2013} Munteanu, O.; Sesum, N. \emph{On gradient Ricci solitons}. J. Geom. Anal. \textbf{23} (2013), no. 2, 539–561.

\bibitem{MSW2019} Munteanu, O.; Sung, C. J. A.; Wang, J. \emph{Poisson equation on complete manifolds}. Adv. Math. \textbf{348} (2019), 81–145.

\bibitem{MSW21} Munteanu, O.; Schulze, F.; Wang, J. \emph{Positive solutions to Schr\"odinger equations and geometric applications}. J. Reine Angew. Math. \textbf{774} (2021), 185–217.

\bibitem{MW2015} Munteanu, O.; Wang, J. \emph{Topology of K\"ahler Ricci solitons}. J. Differential Geom. \textbf{100} (2015), no. 1, 109–128.

\bibitem{MW2015b} Munteanu, O.; Wang, J. \emph{Geometry of shrinking Ricci solitons}. Compos. Math. \textbf{151} (2015), no. 12, 2273–2300.

\bibitem{MW2017a} Munteanu, O.; Wang, J. \emph{Positively curved shrinking Ricci solitons are compact}. J. Differential Geom. \textbf{106} (2017), no. 3, 499–505.

\bibitem{MW2017b} Munteanu, O.; Wang, J. \emph{Conical structure for shrinking Ricci solitons}. J. Eur. Math. Soc. (JEMS) \textbf{19} (2017), no. 11, 3377–3390.

\bibitem{MW2019} Munteanu, O.; Wang, J. \emph{Structure at infinity for shrinking Ricci solitons}. Ann. Sci. \'Ec. Norm. Sup\'er. (4) \textbf{52} (2019), no. 4, 891–925.

\bibitem{MW2022} Munteanu, O.; Wang, J. \emph{Ends of gradient Ricci solitons}. J. Geom. Anal. \textbf{32} (2022), no. 12, Paper No. 303, 26 pp.

\bibitem{MW11} Munteanu, O.; Wang, M.-T. \emph{The curvature of gradient Ricci solitons}. Math. Res. Lett. \textbf{18} (2011), no. 6, 1051–1069.

\bibitem{Naber10} Naber, A. \emph{Noncompact shrinking four solitons with nonnegative curvature}. J. Reine Angew. Math. \textbf{645} (2010), 125–153.

\bibitem{Perelman02} Perelman, G. \emph{The entropy formula for the Ricci flow and its geometric applications}. arXiv:math.DG/0211159 (2002).

\bibitem{Perelman03} Perelman, G. \emph{Ricci flow with surgery on three-manifolds}. arXiv:math.DG/0303109 (2003).

\bibitem{PW20} Petersen, P.; Wink, M. \emph{The Bochner technique and weighted curvatures}. SIGMA Symmetry Integrability Geom. Methods Appl. \textbf{16} (2020), Paper No. 058.

\bibitem{PW21a} Petersen, P.; Wink, M. \emph{New curvature conditions for the Bochner technique}. Invent. Math. \textbf{224} (2021), 33–54.

\bibitem{PW21b} Petersen, P.; Wink, M. \emph{Vanishing and estimation results for Hodge numbers}. J. Reine Angew. Math. \textbf{780} (2021), 197–219.

\bibitem{PRS11} Pigola, S.; Rimoldi, M.; Setti, A. G. \emph{Remarks on non-compact gradient Ricci solitons}. Math. Z. \textbf{268} (2011), no. 3-4, 777–790.

\bibitem{QW2022} Qu, Y.; Wu, G. \emph{When does gradient Ricci soliton have one end?}. Ann. Global Anal. Geom. \textbf{62} (2022), no. 3, 679–691.

\bibitem{RS78} Reed, M.; Simon, B. \emph{Analysis of operators}. Methods of Modern Mathematical Physics, IV. Academic Press, New York, 1978.

\bibitem{Ri14} Rimoldi, M. \emph{On a classification theorem for self-shrinkers}. Proc. Amer. Math. Soc. \textbf{142} (2014), no. 10, 3605–3613.

\bibitem{Sal1992} Saloff-Coste, L. \emph{Uniformly elliptic operators on Riemannian manifolds}. J. Differential Geom. \textbf{36} (1992), no. 2, 417–450.

\bibitem{Schwarz95} Schwarz, G. \emph{Hodge decomposition---a method for solving boundary value problems}. Lecture Notes in Math., 1607, Springer-Verlag, Berlin, 1995.

\bibitem{Sesum06} Sesum, N. \emph{Convergence of the Ricci flow toward a unique soliton}. Comm. Anal. Geom. \textbf{14} (2006), no. 2, 283–343.

\bibitem{SunZhang24} Sun, S.; Zhang, J. \emph{K\"ahler–Ricci shrinkers and Fano fibrations}. arXiv:2410.09661 (2024).

\bibitem{Wyl08} Wylie, W. \emph{Complete shrinking Ricci solitons have finite fundamental group}. Proc. Amer. Math. Soc. \textbf{136} (2008), no. 5, 1803–1806.

\bibitem{Ye04} Yeganefar, N. \emph{Sur la $L^2$-cohomologie des vari\'et\'es \`a courbure n\'egative}. Duke Math. J. \textbf{122} (2004), no. 1, 145–180.

\end{thebibliography}
\end{document}